\documentclass{siamltex1213}

\usepackage{amsmath,amssymb,amsfonts,cmmib57,mathdots,xcolor,mathrsfs}
\usepackage[latin1]{inputenc}
\usepackage{enumerate}

\newcommand{\rev}{\ensuremath\mathrm{rev}}

\newcommand{\FF}{{\mathbb F}}

\DeclareMathOperator{\vect}{vec}

\newtheorem{remark}[theorem]{Remark}
\newtheorem{example}[theorem]{Example}

\title{Structured backward error analysis of linearized structured polynomial eigenvalue problems}
\author{Froil\'an M. Dopico\footnotemark[1], Javier P\'{e}rez\footnotemark[2], and Paul Van Dooren \footnotemark[3]}
\begin{document}
\maketitle
\slugger{simax}{xxxx}{xx}{x}{x--x}%slugger should be set to mms, siap, sicomp, sicon, sidma, sima, simax, sinum, siopt, sisc, or sirev
\renewcommand{\thefootnote}{\fnsymbol{footnote}}
\footnotetext[1]{Departamento de Matem\'{a}ticas, Universidad Carlos III de Madrid,  Avenida de la Universidad 30, 28911, Legan\'{e}s, Spain. Email: {\tt dopico@math.uc3m.es}.
Supported by ``Ministerio de Econom\'{i}a, Industria y Competitividad of Spain" and ``Fondo Europeo de Desarrollo Regional (FEDER) of EU" through grants MTM-2015-68805-REDT, MTM-2015-65798-P (MINECO/FEDER, UE).}
\footnotetext[2]{Department of Computer Science, KU Leuven, Celestijnenlaan 200A, 3001 Heverlee, Belgium. Email: {\tt javierperez@kuleuven.be}. Partially supported by KU Leuven Research Council grant OT/14/074 and the Interuniversity Attraction Pole DYSCO, initiated by the Belgian State Science Policy Office.}
\footnotetext[3]{Department of Mathematical Engineering, Universit\'e catholique de Louvain, Avenue Georges
Lema\^itre 4, B-1348 Louvain-la-Neuve, Belgium. Email: {\tt paul.vandooren@uclouvain.be}. Supported by the Belgian network DISCO (Dynamical Systems, Control, and Optimization), funded by the Interuniversity Attraction Poles Programme initiated by the Belgian Science Policy Office.}

\renewcommand{\thefootnote}{\arabic{footnote}}

\begin{abstract}
We start by introducing a new class of structured matrix polynomials, namely, the class of $\mathbf{M}_A$-structured matrix polynomials, to provide a common framework for many classes of structured matrix polynomials that are important in applications: the classes of (skew-)symmetric, (anti-)palindromic, and alternating matrix polynomials.
Then, we introduce the families of $\mathbf{M}_A$-structured strong block minimal bases pencils and of $\mathbf{M}_A$-structured block Kronecker pencils, which are particular examples of block minimal bases pencils recently introduced by Dopico, Lawrence, P\'erez and Van Dooren, and show that any $\mathbf{M}_A$-structured odd-degree matrix polynomial can be strongly linearized via an $\mathbf{M}_A$-structured block Kronecker pencil. 
Finally, for the  classes of \break (skew-)symmetric, (anti-)palindromic, and alternating odd-degree matrix polynomials, the $\mathbf{M}_A$-structured framework allows us to perform a global and structured backward stability analysis of complete structured polynomial eigenproblems, regular or singular, solved by applying to a $\mathbf{M}_A$-structured block Kronecker pencil a structurally backward stable algorithm that computes its complete eigenstructure, like  the palindromic-QR algorithm or the structured  versions of the staircase algorithm.
This analysis allows us to identify those $\mathbf{M}_A$-structured block Kronecker pencils that yield a computed complete eigenstructure which is the exact one of a slightly perturbed structured matrix polynomial.
These pencils include (modulo permutations) the well-known block-tridiagonal and block-antitridiagonal structure-preserving linearizations. 
Our analysis incorporates structure to the recent (unstructured) backward error analysis performed for block Kronecker linearizations by Dopico, Lawrence, P\'erez and Van Dooren, and share with it its key features, namely, it is a rigorous analysis valid for finite perturbations, i.e., it is not a first order analysis, it provides precise bounds, and it is valid simultaneously for a
large class of structure-preserving strong linearizations. 
\end{abstract}

\pagestyle{myheadings}
\thispagestyle{plain}
%\markboth{J. P\'{E}REZ}{}

\section{Introduction}\label{sec:intro}

Matrix polynomials with special algebraic structures occur in numerous applications  in engineering, mechanics, control, linear systems theory, and computer-aided graphic design.
Some of the most common of these algebraic structures that appear in applications are the (skew-)symmetric \cite{symmetric,Skew}, (anti-)palindromic \cite{GoodVibrations,Palindromic}, and alternating structures \cite{GoodVibrations,Alternating}.
Palindromic matrix polynomials appear, to name a few applications, in the study of  resonance phenomena of  rail tracks under high frequency excitation forces \cite{Jacobi,rail1,rail2}, in the numerical simulation of the behavior of periodic surface acoustic wave filters \cite{SAW1,SAW2}, in passivity tests of a linear dynamical system \cite{passivity}, and in discrete-time linear-quadratic optimal control problems \cite{optimal_control}.
Symmetric (or Hermitian) matrix polynomials arise in the classical problem of vibration analysis \cite{Hermitian1,Lancaster_book,Lancaster_book2,quadratic}, and alternating matrix polynomials find applications, for instance, in the study of corner singularities in anisotropic elastic materials \cite{corner1,corner2,corner3}, in the study of gyroscopic systems \cite{giro1,Lancaster_book2,Lancaster_paper}, and in continuous-time linear-quadratic optimal control problems \cite{optimal_control}.
Further details of different applications of (structured and unstructured) matrix polynomials can be found in  the classical references \cite{Lancaster_book,Kailath,Rosenbrock}, the modern surveys \cite[Chapter 12]{Volker_book} and \cite{NEV,quadratic}, and the references therein, and in the reference \cite{GoodVibrations}.

Structured matrix polynomials present rich symmetries in their spectra, which are discussed in detail, for example, in \cite{spectral_equivalence,Alternating,Palindromic,Skew}.
Since the algebraic structures of matrix polynomials stem usually from the physical symmetries underlying problems arising from applications, these spectral symmetries reflect specific physical properties, and it is very desirable to preserve them in computed solutions.
However, general unstructured polynomial eigensolvers may destroy these spectral symmetries due to rounding errors.
As a consequence, the development and investigation of polynomial eigensolvers that are able to exploit and preserve the structure that the matrix polynomials might possess, have been the focus of an intense research during the last decade (see, for example \cite[Chapters 1, 2, 3, and 12]{Volker_book}, and the references therein).

Square regular matrix polynomials are usually related to {\em polynomial eigenvalue problems} (PEPs), while singular matrix polynomials are related to {\em complete polynomial eigenvalue problems} (CPEs), since in the singular case the so called minimal indices
have to be considered in addition to the eigenvalues.
When the spectral symmetries of structured matrix polynomials are taken into account (i.e., they have to be preserved in the computed solution), those problems receive the names of {\em structured polynomial eigenvalue problems} (SPEPs) and {\em structured complete polynomial eigenvalue problems} (SCPEs), respectively.
The standard approach to solve a PEP or a CPE (or a SPEP or a SCPE) associated with a matrix polynomial $P(\lambda)$ is to linearize $P(\lambda)$ into a matrix pencil (i.e., a matrix polynomial of degree 1).
Linearization transforms the original polynomial eigenvalue problem into an equivalent generalized eigenvalue problem, which can be solved by using mature and well-understood generalized eigensolvers such as the QZ algorithm  or the staircase algorithm \cite{QZ,staircase,VanDooren83} or their structured counterparts \cite{skew-staircase,Implicit_palQR,MMMM_pal,Schroder_thesis,palQR}.
For this reason, one of the preferred approaches to develop structured numerical methods for solving SPEPs and SCPEs associated with structured matrix polynomials starts by devising structure-preserving linearizations \cite{Greeks2,FPR1,FiedlerHermitian,FPR2,FPR3,VectorSpaces,
PalindromicFiedler,PartI,PartII,PartIII,symmetric,definite,
ChebyshevPencils,GoodVibrations,ChebyFiedler,Leo2016}.

The theory of linearizing structured matrix polynomials in a structure preserving way is already well-understood \cite{spectral_equivalence,Alternating,Palindromic,Skew}.
It is well-known that any odd-degree structured matrix polynomial in the classes listed in the first paragraph of this section can be linearized in a structure-preserving way, regardless of whether the matrix polynomial is regular or singular.
However, some even-degree structured matrix polynomials in these classes do not have any linearization with the same structure due to some spectral subtle obstructions \cite[Section 7.2]{spectral_equivalence}.
This phenomenon suggests that for even-degree structured matrix polynomials linearizations should sometimes be replaced by other low degree matrix polynomials in numerical computations \cite{pal_quadratification}.
Due to this even-degree/odd-degree dichotomy for the existence of classes of structure-preserving linearizations, we only consider in this work numerical methods based on structure-preserving linearizations for solving SPEPs or SCPEs associated with odd-degree matrix polynomials.

One interesting recent advance in the theory of linearizations of matrix polynomials has been the introduction of the family of (strong) block minimal bases pencils \cite{minimal_pencils}, since many of the linearizations that have appeared previously in the literature are included in this family of pencils \cite{Maribel} and, in addition, allow a simple, concise, and unified theory \cite{minimal_pencils}. 
A particular but very important subfamily of strong block minimal bases pencils is the family of block Kronecker pencils \cite{minimal_pencils}.
Block Kronecker pencils include (modulo permutations) all Fiedler linearizations \cite{Fiedler_pencils,minimal_pencils}, but infinitely many more linearizations are also included in this family.
All the linearizations belonging to the family of block Kronecker pencils have the following properties that are very desirable in numerical applications:
\begin{itemize}
\item[\rm (i)] they are strong linearizations, regardless whether the matrix polynomial is regular or singular;
\item[\rm (ii)] they are easily constructible from the coefficients of the matrix polynomials;
\item[\rm (iii)] eigenvector of regular matrix polynomials are easily recovered from those of the linearizations;
\item[\rm (iv)] minimal bases of singular matrix polynomials are easily recovered from those of the linearizations;
\item[\rm (v)] there exists a simple shift relation between the minimal indices of singular matrix polynomials and the minimal indices of the linearizations, and such relation is robust under perturbations;
\item[\rm (vi)] they guarantee global backward stability of polynomial eigenvalue problems solved via block Kronecker linearizations.
\end{itemize}
Additionally, block Kronecker pencils have been generalized to allow one to construct strong linearizations for matrix polynomials that are expressed in some non-monomial polynomial bases \cite{ChebyshevPencils,Leo2016}.

Another key advantage of the family of strong block minimal bases pencils is that one can find easily in it structure-preserving strong linearizations for odd-degree structured matrix polynomials in relevant structured classes \cite[Section 5]{Leo2016}.
This observation has led to the introduction of the family of structured block Kronecker pencils \cite{PartI,PartII,PartIII}. 
Linearizations based on structured block Kronecker pencils share with  block Kronecker linearizations  properties (i)--(vi), listed above, together with the property that they preserve a number of important structures that an odd-degree matrix polynomial might possess.

Once a structured matrix polynomial $P(\lambda)$ is linearized  via a structure-preserving strong linearization $\mathcal{L}(\lambda)$, a structured method  (i.e., a method preserving the spectral symmetries of the spectrum of the polynomial) can be applied on the pencil $\mathcal{L}(\lambda)$ to solve the SPEP or SCPE associated with $P(\lambda)$. 
There are many available structure-preserving methods for computing the eigenstructure of certain structured matrix pencils.
For example, for regular palindromic or anti-palindromic matrix pencils we have a URV-like method \cite{URV}, a Jacobi-like method \cite{Jacobi}, the palindromic-QR algorithm \cite{Implicit_palQR,palQR}, doubling methods \cite{doubling},  or the QZ algorithm with the Laub trick \cite{MMMM_pal}.
For singular palindromic or anti-palindromic matrix pencils there is a structured version of the GUPTRI algorithm (a structured  staircase form), that deflates the singular part of palindromic pencils \cite{Schroder_thesis}.
All these methods can also be applied to alternating matrix pencils as well, since any alternating pencil can be transformed into a palindromic or anti-palindromic pencil via a Cayley transformation \cite{GoodVibrations}.
Methods for other structures can be found in \cite{Jacobi_Hermitian,Krylov_symmetric}, for example.

Some of the structured methods for structured pencils mentioned in the paragraph above are \emph{structurally global backward stable}\footnote{Structurally global backward stable algorithms are called strongly backward stable algorithms in \cite{Implicit_palQR,Schroder_thesis}.} \cite{Implicit_palQR,Schroder_thesis}, and others behave in practice in a structurally global backward stable way. 
This means that if the complete eigenstructure of a structured matrix polynomial is computed as the complete eigenstructure of a structure-preserving linearization of the matrix polynomial, then the computed complete eigenstructure is the exact one of a nearby matrix pencil with the same structure as the given matrix polynomial.
However, it has been an open problem to determine whether or not these methods compute the exact complete eigenstructure of a structured nearby matrix polynomial.
We only know one reference where this problem is addressed in the case of skew-symmetric matrix polynomials \cite{Skew-symmetric_Andrii}.
Nonetheless, the analysis in \cite{Skew-symmetric_Andrii} is only valid for infinitesimal perturbations and it does not provide precise bounds.
Only precise ``local'' structured backward error analyses valid for
each particular computed eigenvalue or eigenpair have been developed so far.
See, for example, \cite{Ahmad2011,Bora2014,Bora_her,Bora_pal}, or \cite{Adhikari,Rafikul2011} for the case of the structured linearizations in the vector spaces $\mathbb{L}_1(P)$, $\mathbb{L}_2(P)$ and $\mathbb{DL}(P)$, introduced in \cite{MMMM_vector_space,GoodVibrations} and \cite{symmetric}.

The main goal of this work is to perform for the first time a rigorous structured global backward error analysis of SPEPs or SCPEs associated with odd-degree structured matrix polynomials  of certain important classes solved by applying a structured algorithm to a structured block Kronecker linearization.
The backward error analysis that we present here takes its inspiration from the (unstructured) global backward error analysis of PEPs and CPEs solved via block Kronecker linearizations  performed in \cite[Section 6]{minimal_pencils}.
As a consequence, our error analysis shares with the analysis in  \cite[Section 6]{minimal_pencils} its novel properties with respect to previous global backward error analyses: (1) it is valid for perturbations with finite norms, (2) it delivers precise bounds, (3) it is valid simultaneously for a very large class of structure-preserving linearizations.
As a corollary of our results, we solve the open problem of proving that the famous block-tridiagonal and block-antitridiagonal structure preserving strong linearizations presented in \cite{Greeks2,PalindromicFiedler,Alternating,Palindromic,Skew} yield computed complete eigenstructures of structured matrix polynomials that enjoy perfect structured backward stability
from the polynomial point of view.

The rest of the paper is organized as follows.
In Section \ref{sec:basics}, we review some basic concepts and results, and summarize the notation used through the paper.
In Section \ref{sec:Mobius}, we recall M\"obius transformations of matrix polynomials and their relation with structured matrix polynomials.
The concept of  $\mathbf{M}_A$-structured matrix polynomial is also introduced in this section with the aim of providing a common framework for the classes of (skew-)symmetric, (anti-)palindromic and alternating matrix polynomials of odd degree.
In Section \ref{sec:minimal_bases_pencils}, we recall the family of (strong) block minimal bases pencils, and state some of its most important properties.
We also introduce the new family of $\mathbf{M}_A$-structured block minimal bases pencils, which is a subfamily of strong block minimal bases pencils, and use this family to show that any odd-degree $\mathbf{M}_A$-structured matrix polynomial can be strongly linearized in a structure-preserving way.
In Section \ref{sec:classical_structures}, we introduce the family of $\mathbf{M}_A$-structured block Kronecker pencils, review the family of structured block Kronecker pencils, and review how structure-preserving strong linearizations for odd-degree (skew-)symmetric, (anti-)palindromic or alternating matrix polynomials can be easily constructed from structured block Kronecker pencils.
Finally, in Section \ref{sec:analysis}, we perform a rigorous structured and global backward error analysis of SPEPs or SCPEs solved by means of structured block Kronecker pencils.
Our conclusions are presented in Section \ref{sec:conclusions}.

\section{Basic concepts, auxiliary results and notation}\label{sec:basics}

Throughout the paper we use the following notation.
By $\mathbb{F}$ we denote either the field of complex numbers $\mathbb{C}$ or the field of real numbers $\mathbb{R}$.
We also consider the involution $a\rightarrow \overline{a}$, that is, the identity map when $\mathbb{F}=\mathbb{R}$, or, when $\mathbb{F}=\mathbb{C}$,  the bijection that maps any complex number to its complex conjugate.
 By $\mathbb{F}(\lambda)$ and $\mathbb{F}[\lambda]$ we denote, respectively, the field of rational functions and the ring of polynomials with coefficients in $\mathbb{F}$.
The set of $m\times n$ matrices with entries in $\mathbb{F}[\lambda]$ is denoted by $\mathbb{F}[\lambda]^{m\times n}$. 
Usually, we refer to this set as the set of $m\times n$ \emph{matrix polynomials}, and any $P(\lambda)\in \mathbb{F}[\lambda]^{m\times n}$ is called an $m\times n$ matrix polynomial.
Row and column \emph{vector polynomials} refer to matrix polynomials with $m=1$ or $n=1$, respectively.
The set of $m\times n$ matrices with entries in $\mathbb{F}(\lambda)$ is denoted by $\mathbb{F}(\lambda)^{m\times n}$. 
The algebraic closure of the field $\mathbb{F}$ is denoted by $\overline{\mathbb{F}}$.

A matrix polynomial $P(\lambda)\in \mathbb{F}[\lambda]^{m\times n}$ is said to have \emph{grade} $g$ if it is written as
\begin{equation}\label{eq:poly}
P(\lambda) = P_g\lambda^g + P_{g-1}\lambda^{g-1} + \cdots + P_1\lambda + P_0, \quad \mbox{with }P_g,\hdots,P_0\in\mathbb{F}^{m\times n},
\end{equation}
where any of the coefficient matrices $P_i$, including the leading coefficient $P_g$, may be the zero matrix. 
The \emph{degree} of the matrix polynomial \eqref{eq:poly} is denoted by $\deg(P(\lambda))$, and it refers to the maximum integer $d$ such that $P_d$ is a nonzero matrix.
Notice that a polynomial of degree $d$ can be considered as a polynomial of grade $g\geq d$. 
In this work, when the grade of a polynomial is not explicitly stated, we consider its grade as the degree of the polynomial.

For any $g\geq \deg(P(\lambda))$, the \emph{$g$-reversal matrix polynomial} of $P(\lambda)$ is the matrix polynomial
\[
\rev_g P(\lambda):= \lambda^g P(\lambda^{-1}).
\]
Notice that the $g$-reversal operation maps matrix polynomials of grade $g$ to matrix polynomials with the same grade. 
However, the degree of $\rev_g P(\lambda)$ may be different to the degree of $P(\lambda)$.

The \emph{normal rank} of a matrix polynomial $P(\lambda)\in\mathbb{F}[\lambda]^{m\times n}$ is defined as the rank of $P(\lambda)$ over the field $\mathbb{F}(\lambda)$, and it is denoted by $\rank (P)$.
In other words, the normal rank of $P(\lambda)$ is the size of the largest non-identical zero minor of $P(\lambda)$ (see \cite{Gantmacher}, for example).
By $\rank(P(\lambda_0))$ we refer to the rank of the constant matrix $P(\lambda_0)$ obtained by evaluating the matrix polynomial $P(\lambda)$ at $\lambda_0$.
We say that $P(\lambda_0)$ has full row (resp. column) rank if $\rank (P(\lambda_0))=m$ (resp. $\rank (P(\lambda_0))=n$). 

The operator $(\cdot)^\star$ denotes either the transpose when $\mathbb{F}=\mathbb{R}$ or the conjugate transpose when $\mathbb{F}=\mathbb{C}$. 
Given a matrix polynomial $P(\lambda)$ as in \eqref{eq:poly}, the matrix polynomials $P(\lambda)^\star$ and $\overline{P}(\lambda)$ are defined as $P(\lambda)^\star := P_g^\star\lambda^g+\cdots + P_1^\star \lambda + P_0^\star$ and $\overline{P}(\lambda) = \overline{P}_g \lambda^g+\cdots+\overline{P}_1\lambda+\overline{P_0}$, respectively, where the conjugate of a matrix should be understood entrywise.

We focus in this work mainly on square matrix polynomials (that is, $m=n$) with one of the following algebraic structures:

\medskip

\begin{itemize}
\item[\rm (i)] \emph{$\star$-symmetric}: $P(\lambda)^\star = P(\lambda)$,
\item[\rm (ii)] \emph{$\star$-skew-symmetric}: $P(\lambda)^\star =-P(\lambda)$,
\item[\rm (iii)] \emph{$\star$-palindromic}: $P(\lambda)^\star = \rev_g P(\lambda)$,
\item[\rm (iv)] \emph{$\star$-anti-palindromic}: $P(\lambda)^\star = -\rev_g P(\lambda)$,
\item[\rm (v)] \emph{$\star$-even}: $P(\lambda)^\star = P(-\lambda)$,
\item[\rm (vi)]  \emph{$\star$-odd}: $P(\lambda)^\star = -P(-\lambda)$,
\end{itemize}

\medskip

\noindent where $g$ denotes the grade of $P(\lambda)$.
 A matrix polynomial is said to be \emph{$\star$-alternating} if it is either $\star$-even or $\star$-odd.
 When $\mathbb{F}=\mathbb{C}$, a $\star$-(skew-)symmetric matrix polynomial is usually called a (skew-)Hermitian matrix polynomial \cite{Hermitian}. 
 However, we do not employ that terminology  in this paper.
Also, most of the times we drop the ``$\star$-'' in the notation, and just say (skew-)symmetric, (anti-)palindromic or alternating matrix polynomials.
 Additionally, we denote  by $\mathscr{S}(P)\in\{$symmetric, skew-symmetric, palindromic, anti-palindromic, even, odd$\}$ the structure that the structured matrix polynomial $P(\lambda)$  posses.
 
An important distinction in the theory of matrix polynomials is between regular and singular matrix polynomials. 
A matrix polynomial $P(\lambda)$ is said to be \emph{regular} if it is square and the scalar polynomial $\det P(\lambda)$ is not identically equal to the zero polynomial. 
Otherwise, the matrix polynomial $P(\lambda)$ is said to be \emph{singular}.
The \emph{complete eigenstructure} of a regular matrix polynomial consists of its elementary divisors (spectral structure), both finite and infinite, while for a singular matrix polynomial it consists of its elementary divisors together with its right and left minimal indices (spectral structure+singular structure).
The singular structure of matrix polynomials will be briefly reviewed later in the paper.
For more detailed definitions of the spectral  structure of matrix polynomials, we refer the reader to \cite[Section 2]{spectral_equivalence}.

An important feature of  structured matrix polynomials  are the special symmetry properties of their spectral \cite{symmetric,GoodVibrations,Alternating,Palindromic,Skew} and singular structures \cite{singular}.
As we mentioned in the introduction, the problem of computing the complete eigenstructure of a structured matrix polynomial using an algorithm that preserves its spectral and singular structure symmetries in the computed solution is called in this work the \emph{structured polynomial eigenvalue problem (SPEP)}, for regular matrix polynomials, or the \emph{structured complete polynomial eigenvalue problem (SCPE)}, for singular matrix polynomials.

Minimal bases and minimal indices play a relevant role in this work, so they are reviewed in the following.
When a matrix polynomial $P(\lambda)\in\mathbb{F}[\lambda]^{m\times n}$ is singular, it has nontrivial left and/or right \emph{rational null spaces}
\begin{equation} \label{eq:nullspaces}
\begin{split}
\mathcal{N}_\ell(P) & := \{y(\lambda)^T\in\mathbb{F}(\lambda)^{1\times m} \quad \mbox{such that} \quad y(\lambda)^TP(\lambda) = 0\},\\
\mathcal{N}_r(P) & := \{x(\lambda)\in\mathbb{F}(\lambda)^{n \times 1} \quad \mbox{such that} \quad P(\lambda)x(\lambda) = 0\}.
\end{split}
\end{equation}
These two spaces are particular instances of a \emph{rational} subspace \cite{Forney}.
Any rational subspace  $\mathcal{V}$ has always bases consisting entirely of vector polynomials.
The \emph{order} of a vector polynomial basis of $\mathcal{V}$ is defined as the sum of the degrees of its vectors \cite[Definition 2]{Forney}.
The \emph{minimal bases} of $\mathcal{V}$ are those polynomial bases of $\mathcal{V}$ with least order \cite[Definition 3]{Forney}. 
Although minimal bases are not unique,  the ordered list of degrees of the vector polynomials in any minimal basis of $\mathcal{V}$ is always the same \cite[Remark 4, p. 497]{Forney}. 
This list of degrees is called the list of {\em minimal indices} of $\mathcal{V}$.
Then, the \emph{left (resp. right) minimal indices and bases of a matrix polynomial} $P(\lambda)$ are defined as those of the rational subspace $\mathcal{N}_\ell(P)$ (resp. $\mathcal{N}_r(P)$).

To work in practice with minimal bases the following definition will be useful, where by the \emph{$i$th row degree} of a matrix polynomial $Q(\lambda)$ we denote the degree of the $i$th row of $Q(\lambda)$.
\begin{definition}{\rm \cite[Definition 2.3]{zigzag}} \label{def:rowreduced}
Let $Q(\lambda)\in\mathbb{F}[\lambda]^{m\times n}$ be a matrix polynomial with row degrees $d_1,d_2,\hdots,d_m$. 
The {\em highest row degree coefficient matrix} of $Q(\lambda)$, denoted by $Q_h$, is the $m \times n$ constant matrix whose $j$th row is the coefficient of $\lambda^{d_j}$ in the $j$th row of $Q(\lambda)$, for $j=1,2,\hdots,m$. 
The matrix polynomial $Q(\lambda)$ is called {\em row reduced} if $Q_h$ has full row rank.
\end{definition}

Theorem \ref{thm:minimal_basis} is a useful characterization of minimal bases.
This theorem can be found in, for example,  \cite[Main Theorem-Part 2, p. 495]{Forney}.
However, for convenience, we present here the version in less abstract terms  in  \cite[Theorem 2.14]{FFP2015}.
\begin{theorem}
\label{thm:minimal_basis}
The rows of a matrix polynomial $Q(\lambda)\in\mathbb{F}[\lambda]^{m\times n}$ are a minimal basis of the rational subspace they span if and only if $Q(\lambda_0) \in \overline{\mathbb{F}}^{m \times n}$ has full row rank for all $\lambda_0 \in \overline{\mathbb{F}}$ and $Q(\lambda)$ is row reduced.
\end{theorem}

\begin{remark} Since all of the minimal bases that appear in this work are arranged as the rows of a matrix, with a slight abuse of notation, we say that an $m\times n$ matrix polynomial (with $m < n$) is a minimal basis if its rows form a minimal basis of the rational subspace they span.
\end{remark}

Another fundamental concept in this paper is the concept of \emph{dual minimal bases}, which is introduced in Definition \ref{def:dualminimalbases}. 
\begin{definition}{\rm (see \cite{Kailath} or \cite[Definition 2.10]{zigzag})} \label{def:dualminimalbases}
Two matrix polynomials $K(\lambda)\in\FF[\lambda]^{m_1\times n}$ and $N(\lambda)\in\FF[\lambda]^{m_2\times n}$ are called \emph{dual minimal bases} if $K(\lambda)$ and $N(\lambda)$ are both minimal bases and they satisfy
$m_1+m_2 = n$ and $K(\lambda)N(\lambda)^T = 0$.
\end{definition}
\begin{remark}
Following the convention in \cite{minimal_pencils}, we will sometimes say ``$N(\lambda)$ is a minimal basis dual to $K(\lambda)$'', or vice versa, to refer to matrix polynomials $K(\lambda)$ and $N(\lambda)$ as those in Definition \ref{def:dualminimalbases}.
\end{remark}

 We illustrate in Example \ref{ex-L-Lamb} the concept of dual minimal bases with a simple example that plays a key role in this paper (this example can be also found in \cite[Example 2.6]{minimal_pencils}).
Here and throughout the paper we occasionally omit some, or all, of the zero entries of a matrix. 
\begin{example} \label{ex-L-Lamb}   Consider the following matrix polynomials:
\begin{equation}
\label{eq:Lk}
L_k(\lambda):=\begin{bmatrix}
-1 & \lambda  \\
& -1 & \lambda \\
& & \ddots & \ddots \\
& & & -1 & \lambda  \\
\end{bmatrix}\in\mathbb{F}[\lambda]^{k\times(k+1)},
\end{equation}
and
\begin{equation}
  \label{eq:Lambda}
  \Lambda_k(\lambda)^T :=
\begin{bmatrix}
      \lambda^{k} & \cdots & \lambda & 1
\end{bmatrix} \in \FF[\lambda]^{1\times (k+1)}.
\end{equation}
Using Theorem \ref{thm:minimal_basis}, it is easily checked that $L_k(\lambda)$ and $\Lambda_k(\lambda)^T$ are both minimal bases. 
Additionally, $L_k(\lambda)\Lambda_k(\lambda)=0$ holds. 
Therefore, $L_k(\lambda)$ and $\Lambda_k(\lambda)^T$ are dual minimal bases. 
Also, from \cite[Corollary 2.4]{minimal_pencils} and basic properties of the Kronecker product $\otimes$, we get that $L_k(\lambda) \otimes I_n$ and $\Lambda_k(\lambda)^T \otimes I_n$ are also dual minimal bases.
\end{example}

Notice the following  property of the matrix polynomials $L_k(\lambda)\otimes I_n$ and $\Lambda_k(\lambda)^T\otimes I_n$ in Example \ref{ex-L-Lamb}.
Both are minimal bases whose row degrees are all equal (equal to 1 in the case of $L_k(\lambda)\otimes I_n$, and equal to $k$ in the case of $\Lambda_k(\lambda)^T\otimes I_n$).
Those are the minimal bases that we are interested in this work, and, sometimes, we will refer to them as \emph{constant-row-degrees minimal bases}. 

In Lemma \ref{lemma:constant-row-degrees} we present a  simple characterization of constant-row-degrees minimal bases.
This result is an immediate corollary of Theorem \ref{thm:minimal_basis}, together with the obvious fact that if the leading coefficient of a matrix polynomial has full row rank, then its leading and highest row degree coefficients coincide, so its proof is omitted.
\begin{lemma}\label{lemma:constant-row-degrees}
The matrix polynomial $K(\lambda)=\sum_{i=0}^\ell K_i\lambda^i$ of degree $\ell$ is a constant-row-degrees minimal basis if and only if $K(\lambda_0)$ has full row rank for all $\lambda_0\in\overline{\mathbb{F}}$ and its leading coefficient $K_\ell$ has full row rank.
\end{lemma}

We now recall the definitions of unimodular matrix polynomials and (strong) linearizations  of matrix polynomials.
A \emph{unimodular matrix polynomial} $U(\lambda)$ is a matrix polynomial whose determinant $\det U(\lambda)$ is a nonzero constant.
A matrix pencil  $\mathcal{L}(\lambda)$ is said to be a linearization of a matrix polynomial $P(\lambda)$ of grade $g$ if for some $s\geq 0$ there exist unimodular matrices $U(\lambda)$ and $V(\lambda)$ such that
\[
U(\lambda)\mathcal{L}(\lambda)V(\lambda) = 
\begin{bmatrix}
I_s & 0 \\
0 & P(\lambda)
\end{bmatrix}.
\]
In addition, a linearization $\mathcal{L}(\lambda)$ is called a \emph{strong linearization} of $P(\lambda)$ if $\rev_1\mathcal{L}(\lambda)$ is a linearization of $\rev_g P(\lambda)$.
We recall that the key property of any strong linearization $\mathcal{L}(\lambda)$ of the matrix polynomial $P(\lambda)$ is that $P(\lambda)$ and $\mathcal{L}(\lambda)$ share the same finite and infinite elementary divisors and the 
same number of left and right minimal indices.
However the minimal indices of $\mathcal{L}(\lambda)$ may take any value \cite[Theorem 4.11]{spectral_equivalence}.
For this reason, in the case
of singular matrix polynomials, the identification of those strong linearizations with the additional property that their minimal indices allow one to recover the minimal indices of the polynomial via some simple rules has been the focus of an intense research \cite{singular,Fiedler_pencils,
PalindromicFiedler,spectral_equivalence,minimal_pencils}.

Given two matrix polynomials $P(\lambda)$ and $Q(\lambda)$ with the same size,   we say that $P(\lambda)$ and $Q(\lambda)$ are \emph{strictly equivalent} if $Q(\lambda)=UP(\lambda)V$, for some nonsingular constant matrices $U$ and $V$, and we say that $P(\lambda)$ and $Q(\lambda)$ are \emph{$\star$-congruent} if $Q(\lambda) = XP(\lambda)X^\star$, for some nonsingular constant matrix $X$.
Clearly, $\star$-congruence is a particular case of strict equivalence.
We recall that strict equivalence preserves both the spectral and singular structures of matrix polynomials \cite[Definition 3.1]{spectral_equivalence}.

Another simple concept that plays an important role in this work is the concept of coninvolutory matrix \cite{coninvolutory}, which is introduced in Definition \ref{def:involutory}
\begin{definition}\label{def:involutory}
A matrix $A\in\mathbb{F}^{n\times n}$ is said to be \emph{coninvolutory} if $A\cdot \overline{A}=I_n$. 
\end{definition}

Coninvolutory matrices when $\mathbb{F}=\mathbb{R}$ are just known as \emph{involutory} matrices, and any real $n\times n$ involutory matrix $A$ satisfies $A\cdot A = I_n$.
In this work, we will make use of $2\times 2$ coninvolutory matrices.
When $\mathbb{F}=\mathbb{R}$,  there is a nice characterization of $2\times 2$ involutory matrices. 
This is shown in Example \ref{ex:involutory}.
\begin{example}\label{ex:involutory}
Any $2\times 2$ real involutory matrix is of the form
\[
\begin{bmatrix}
\pm 1 & 0 \\
0 & \pm 1
\end{bmatrix} \quad \mbox{or} \quad  
\begin{bmatrix}
\pm \sqrt{1-bc} & b \\ c & \mp \sqrt{1-bc}
\end{bmatrix},
\]
where $b,c\in\mathbb{R}$ satisfy $bc\leq 1$.
\end{example}

The backward error analysis in Section \ref{sec:analysis} requires the use of norms of matrix polynomials and their submultiplicative-like properties.
Following \cite{minimal_pencils}, we choose the simple norm in Definition \ref{def:norm}. 
\begin{definition} \label{def:norm} Let $P(\lambda) = \sum_{i=0}^g P_i \lambda^i \in \mathbb{F} [\lambda]^{m\times n}$.
Then the Frobenius norm of $P(\lambda)$ is
\[
\|P(\lambda)\|_F := \sqrt{\sum_{i=0}^g \|P_i\|_F^2} \, .
\]
Notice that the value of the norm $\|P(\lambda)\|_F$ does not depend on the grade chosen for $P(\lambda)$. 
This property allows one to work with $\|P(\lambda)\|_F$ without specifying the grade of the matrix polynomial $P(\lambda)$.
\end{definition}

As it is pointed out in \cite{minimal_pencils}, the norm $\|\cdot\|_F$ is not submultiplicative, that is, $\|P(\lambda) \, Q(\lambda)\|_F\allowbreak \leq \|P (\lambda)\|_F \, \|Q(\lambda)\|_F$ does not hold in general. 
However, Lemma \ref{lemma:normsproducts} shows that the norm $\|\cdot\|_F$ satisfies some submultiplicative-like properties.
%These properties will be useful in Section \ref{sec:analysis}, where  we need to bound the norms of certain products of matrix polynomials.
\begin{lemma} \label{lemma:normsproducts}{\rm \cite[Lemma 2.16]{minimal_pencils}} Let $P(\lambda) = \sum_{i=0}^g P_i \lambda^i$, let $Q(\lambda) = \sum_{i=0}^t Q_i \lambda^i$, and let $\Lambda_k (\lambda)^T$ be the vector polynomial defined in \eqref{eq:Lambda}.
 Then the following inequalities hold:
\begin{enumerate}
\item[\rm (a)] $\displaystyle \|P (\lambda) \, Q(\lambda)\|_F \leq \sqrt{g+1} \cdot \sqrt{\sum_{i=0}^g \|P_i\|_2 ^2}  \cdot \|Q(\lambda)\|_F$  ,

\item[\rm (b)] $\displaystyle \|P(\lambda) \, Q(\lambda)\|_F \leq \sqrt{t+1} \cdot \|P(\lambda)\|_F \cdot \sqrt{\sum_{i=0}^t \|Q_i\|_2 ^2}$ ,

\item[\rm (c)] $\|P (\lambda) \, Q(\lambda)\|_F \leq \, \min\{\sqrt{g+1} , \sqrt{t+1} \}\,  \|P (\lambda)\|_F \, \|Q(\lambda)\|_F$ ,

\item[\rm (d)] $\|P(\lambda) \, (\Lambda_k (\lambda) \otimes I_p) \|_F \leq \, \min\{\sqrt{g+1}, \sqrt{k+1} \} \, \|P(\lambda) \|_F$,

\item[\rm (e)] $\|(\Lambda_k (\lambda)^T \otimes I_n) \, Q(\lambda)\|_F \leq \, \min\{\sqrt{t+1}, \sqrt{k+1} \}  \, \|Q(\lambda) \|_F$,
\end{enumerate}
where we assume that all the products are defined.
\end{lemma}

\medskip

Finally, since in Section \ref{sec:analysis} we need to consider pairs of matrices $(C,D)$ where $C$ and $D$ may have different sizes, and, thus, $(C,D)$ cannot be considered as a matrix pencil, we introduce the corresponding Frobenius norm as:
\begin{equation}\label{eq:norm_pair}
\|(C,D)\|_F:=\sqrt{\|C\|_F^2+\|D\|_F^2}.
\end{equation}

\section{M\"{o}bius transformations and structured odd-grade matrix polynomials}\label{sec:Mobius}

The goal of this section is to introduce a unified framework for the most important classes of structured matrix polynomials of odd grade considered in the literature, namely, (skew-)symmetric, (anti-)palindromic and alternating odd-grade matrix polynomials.
This  requires to introduce the concepts of  M\"{o}bius tranformation and $\mathbf{M}_A$-structured matrix polynomial.

\subsection{M\"obius transformations of matrix polynomials}

M\"obius transformations of matrix polynomials were formally introduced  in \cite{Mobius} as a broader theory for different transformations that had appeared previously in the literature \cite{Hermitian,definite,MMMM_vector_space,
GoodVibrations,Alternating,Palindromic}, and since then, they play an increasingly important role as a useful tool in the theory of matrix polynomials. 

\begin{definition}\label{def:Mobius}{\rm \cite[Definition 3.4]{Mobius}}
Let $A\in {\rm GL}(2,\mathbb{F})$.
The \emph{M\"obius transformation} of $B(\lambda):= \sum_{i=0}^g B_i\lambda^i$ induced by $A$ is  defined by
\[
\mathbf{M}_A[B](\lambda) := \sum_{i=0}^g B_i(a\lambda +b)^i(c\lambda +d)^{g-i}, \quad \mbox{where} \quad A =\begin{bmatrix} a & b \\ c & d \end{bmatrix}.
\]
\end{definition}

We recall that M\"obius transformations are  special cases of rational transformations of matrix polynomials \cite{rational}.
Indeed, a M\"obius transformation can be calculated via the rational expression
\begin{equation}\label{eq:Mobius_rational}
\mathbf{M}_A\left[B \right](\lambda) = 
(c\lambda+d)^g B\left( \frac{a\lambda + b}{c\lambda + d} \right), \quad \mbox{where} \quad A =\begin{bmatrix} a & b \\ c & d \end{bmatrix}.
\end{equation}

In Example \ref{ex:Mobius_pencil}, we illustrate the effect of M\"obius transformation on matrix pencils.
\begin{example}\label{ex:Mobius_pencil}
Let $L(\lambda)=\lambda F+E$ and let $A = \left[\begin{smallmatrix} a & b \\ c & d \end{smallmatrix}\right]\in {\rm GL}(2,\mathbb{F})$.
Then, $\mathbf{M}_A[L](\lambda) = \lambda(aF+cE)+bF+dE$.
\end{example}

The $g$-reversal of a matrix polynomial operation is a well-known example of a M\"obius transformation of matrix polynomials.
We show in Example \ref{ex:reversal} how to formulate this operation as a M\"obius transformation.
\begin{example}\label{ex:reversal}
Let $P(\lambda)=\sum_{i=0}^g P_i\lambda^i$,  and let
$
R_2 := \left[\begin{smallmatrix}
0 & 1 \\
1 & 0
\end{smallmatrix}\right].
$
Then, $\rev_g P(\lambda) = \mathbf{M}_{R_2}[P](\lambda)$.
\end{example}

Many important properties of M\"obius transformations of matrix polynomials follow easily from Definition \ref{def:Mobius} or its rational transformation formulation in \eqref{eq:Mobius_rational}. 
In Proposition \ref{prop:properties_Mobius}, we state without proofs those that will be relevant in this work.
For a thorough study of the properties of M\"obius transformations of matrix polynomials we refer the reader to \cite{Mobius}.
\begin{proposition}\label{prop:properties_Mobius}
For any $A,B\in {\rm GL}(2,\mathbb{F})$ the following statements hold.
\begin{itemize}
\item[\rm (a)] $\mathbf{M}_A[P+Q](\lambda)=\mathbf{M}_A[P](\lambda)+\mathbf{M}_A[Q](\lambda)$, for any $m\times n$ matrix polynomials $P(\lambda)$ and $Q(\lambda)$ both of grade $g$.
\item[\rm (b)] Let $P(\lambda)$ and $Q(\lambda)$ be two matrix polynomials of grades $g_1$ and $g_2$, respectively.
If $P(\lambda)Q(\lambda)$ is defined, then $\mathbf{M}_A[PQ](\lambda)=\mathbf{M}_A[P](\lambda)\mathbf{M}_A[Q](\lambda)$,  where $P(\lambda)Q(\lambda)$ is considered as a matrix polynomial of grade $g_1+g_2$.
\item[\rm (c)] If $Q(\lambda)=P(\lambda)\otimes I_n$, then $\mathbf{M}_A[Q](\lambda) = \mathbf{M}_A[P](\lambda)\otimes I_n$.
\item[\rm (d)] $\mathbf{M}_A[P^T](\lambda)=\mathbf{M}_A[P](\lambda)^T$.
\item[\rm (e)] If $\mathbb{F}=\mathbb{C}$, then  $\overline{\mathbf{M}_A[P]}(\lambda) = \mathbf{M}_{\overline{A}}[\overline{P}](\lambda)$ and $\mathbf{M}_A[P](\lambda)^*=\mathbf{M}_{\overline{A}}[P^*](\lambda)$.
\item[\rm (f)] M\"obius transformations act block-wise, i.e., $\left[\mathbf{M}_A[P](\lambda)\right]_{\mu\kappa} = \mathbf{M}_A[P_{\mu\kappa}](\lambda)$, for any row and column index sets $\mu$ and $\kappa$, and where $[P(\lambda)]_{\mu\kappa}$ has to be considered as a matrix polynomial  with a grade equal to the grade of $P(\lambda)$.
\item[\rm (g)] $\mathbf{M}_B\left[\mathbf{M}_A[P]\right](\lambda) = \mathbf{M}_{AB}[P](\lambda)$.
\end{itemize}
\end{proposition}

\medskip

M\"obius transformations play well with the constant-row-degrees dual minimal bases that will be involved in the construction of the strong block minimal bases pencils in Section \ref{sec:minimal_bases_pencils} (see Definition  \ref{def:minlinearizations}).
More precisely, we have Theorem \ref{thm:minimal_basis_Mobius}.
Some of the results in Theorem \ref{thm:minimal_basis_Mobius} can be obtained from \cite[part (f) of Theorem 7.4]{Mobius}, where the effect of M\"obius transformations on minimal bases is studied.
Nonetheless, for the sake of completeness, we provide a proof of Theorem \ref{thm:minimal_basis_Mobius} here.
\begin{theorem}\label{thm:minimal_basis_Mobius}
Let $A\in{\rm GL}(2,\mathbb{F})$. 
Then, the following statements hold.
\begin{itemize}
\item[\rm (a)] If $K(\lambda)=\sum_{i=0}^\ell K_i\lambda^i$ is a minimal basis with all its row degrees equal to $\ell$, then $\mathbf{M}_A[K](\lambda)$ is also a minimal basis with all its row degrees equal to $\ell$.
\item[\rm (b)] If $K(\lambda)$ and $N(\lambda)$ are a pair of dual minimal bases with all the row degrees of $K(\lambda)$  equal to $\ell$ and all the row degrees of $N(\lambda)$  equal to $t$, then $\mathbf{M}_A[K](\lambda)$ and $\mathbf{M}_A[N](\lambda)$ are also a pair of dual minimal bases with all the row degrees of $\mathbf{M}_A[K](\lambda)$  equal to $\ell$ and all the row degrees of $\mathbf{M}_A[N](\lambda)$  equal to $t$.
\end{itemize}
\end{theorem}
\begin{proof}
Proof of part (a).
In the proof we use the notation $\widehat{K}(\lambda)=\sum_{i=0}^\ell \widehat{K}_i\lambda^i := \mathbf{M}_A[K](\lambda)$, and denote the entries of $A$ as $A=\left[ \begin{smallmatrix} a & b \\ c & d \end{smallmatrix}\right]$.
We first show that $\widehat{K}(\lambda_0)$ has full row rank for all $\lambda_0\in\overline{\mathbb{F}}$. 
The proof proceeds by contradiction.
Assume that $\widehat{K}(\lambda_0)$ is rank deficient for some $\lambda_0$, that is, there exists a vector $x\neq 0$ such that $x^\star\widehat{K}(\lambda_0)=x^\star \sum_{i=0}^\ell K_i (a\lambda_0+b)^i(c\lambda_0+d)^{\ell-i}=0$.
We have to distinguish two cases. 
First, assume that $c\lambda_0+d\neq 0$. 
In this situation we get $x^\star \widehat{K}(\lambda_0)=(c\lambda_0+d)^\ell x^\star K\left( (a\lambda_0+b)/(c\lambda_0+d) \right)=0$ which implies $x^\star K\left( (a\lambda_0+b)/(c\lambda_0+d) \right)=0$, contradicting that $K(\mu_0)$ has full row rank for all $\mu_0$.
Assume, now, that $c\lambda_0+d=0$. 
Notice that the nonsingularity of $A$ implies $a\lambda_0+b\neq 0$.
Then, we get $x^\star \widehat{K}(\lambda_0)=(a\lambda_0+b)^\ell x^\star K_\ell$, which implies $x^\star K_\ell = 0$, contradicting that $K_\ell$ has full row rank. 
Therefore, $\widehat{K}(\lambda_0)$ has full row rank for all $\lambda_0$. 

Next, we proof that $\widehat{K}_\ell$ has full row rank.
Notice that the leading coefficient of $\widehat{K}(\lambda)$ can be computed as $\widehat{K}_\ell = \rev_{\ell} \left[ \mathbf{M}_A[K] \right](0) = \mathbf{M}_{AR_2}[K](0) = \sum_{i=0}^\ell K_i a^ic^{\ell-i}$, where 
$
R_2=\left[\begin{smallmatrix} 0 & 1 \\ 1 & 0 \end{smallmatrix}\right]
$ (recall Example \ref{ex:reversal}).
The proof proceeds by contradiction.
Assume that $\widehat{K}_\ell$ has not full row rank, that is, there exists a vector $x\neq 0$ such that $x^\star \widehat{K}_\ell=x^\star \sum_{i=0}^\ell K_i a^ic^{\ell-i}=0$.
Again, we have to distinguish two cases.
First, assume that $c\neq 0$.
Then, we get $x^\star \widehat{K}_\ell = c^\ell x^\star K(a/c)$, which implies $x^\star K(a/c)=0$, contradicting that $K(\mu_0)$ has full row rank for all $\mu_0$. 
Assume now that $c=0$.
Notice that the nonsingularity of $A$ implies, in this situation, $a\neq 0$.
In this case, we get $x^\star \widehat{K}_\ell=a^\ell x^\star K_\ell=0$  which implies $x^\star K_\ell=0$, contradicting that $K_\ell$ has full row rank. 
Therefore, $\widehat{K}_\ell$ has full row rank. 

Since $\mathbf{M}_A[K](\lambda_0)$ has full row rank for any $\lambda_0\in\overline{\mathbb{F}}$, and its leading matrix coefficient has full row rank, by Lemma \ref{lemma:constant-row-degrees}, we conclude that $\mathbf{M}_A[K](\lambda)$ is a minimal basis with all its row degrees equal to $\ell$.

Proof of part (b).
From part (a) we get that $\mathbf{M}_A[K](\lambda)$ and $\mathbf{M}_A[N](\lambda)$ are minimal bases with all the row degrees of $\mathbf{M}_A[K](\lambda)$  equal to $\ell$ and all the row degrees of $\mathbf{M}_A[N](\lambda)$  equal to $t$.
Then, from $K(\lambda)N(\lambda)^T=0$ together with properties (b) and (d) in Proposition \ref{prop:properties_Mobius}, we get $\mathbf{M}_A[K](\lambda)\mathbf{M}_A[N](\lambda)^T=0$.
Therefore, $\mathbf{M}_A[K](\lambda)$ and $\mathbf{M}_A[N](\lambda)$ are constant-row-degrees dual minimal bases.
\end{proof}

\subsection{$\mathbf{M}_A$-structured matrix polynomials}\label{sec:MA-structure}

M\"obius transformations of matrix polynomials can be used to introduce a new class of structured matrix polynomials that generalizes most of the classes that have been considered in the literature for odd-grade matrix polynomials.
This is done in the following definition, where we introduce the concept of $\mathbf{M}_A$-structured matrix polynomial.
\begin{definition}\label{def:Mstructure}
Let $P(\lambda)=\sum_{i=0}^g P_i\lambda^i \in\mathbb{F}[\lambda]^{n\times n}$ and let $A\in {\rm GL}(2,\mathbb{F})$.
Then, the matrix polynomial $P(\lambda)$ is said to be \emph{$\mathbf{M}_A$-structured} if $\mathbf{M}_A\left[P\right](\lambda)=P(\lambda)^\star$.
\end{definition}

We illustrate in Example \ref{ex:M_A_pencil} the concept of  $\mathbf{M}_A$-structured matrix polynomials in the simplest case, that is, for matrix pencils.
\begin{example}\label{ex:M_A_pencil}
Let $L(\lambda)= \lambda F+E$, and let $A = \left[\begin{smallmatrix} a & b \\ c & d \end{smallmatrix}\right]\in {\rm GL}(2,\mathbb{F})$.
If 
\[
F^\star = aF+cE  \quad \mbox{and} \quad E^\star = bF+dE,
\]
then $L(\lambda)$ is an $\mathbf{M}_A$-structured matrix pencil, and vice versa.
\end{example}

\begin{remark}
The classes of (skew-)symmetric, (anti-)palindromic, and alternating structured odd-grade matrix polynomials are particular examples of $\mathbf{M}_A$-structured matrix polynomials. 
In Table \ref{table:1} we summarize the values of the entries of the matrix $A$ for these structures.
 \begin{table}[ht]
\caption{The classical structured matrix polynomials of odd degree as $\mathbf{M}_A$-structured matrix polynomials: entries of the matrix $A=[a \,\,b;\, c \,\,d]$ for the classes of (skew-)symmetric, (anti-)palindromic, and alternating  matrix polynomials.}
\centering 
\begin{tabular}{|l|c|c|c|c|}
\hline \hline
\bf{Structure} \phantom{\Large{(}} &\phantom{a} $\mathbf{a}$ \phantom{a}&\phantom{a} $\mathbf{b}$ \phantom{a}&\phantom{a} $\mathbf{c}$ \phantom{a} &\phantom{a} $\mathbf{d}$ \phantom{a}\\
\hline\hline
Symmetric  \phantom{\Large{(}}  & 1 & 0 & 0 & 1 \\ 
\hline
Skew-symmetric   \phantom{\Large{(}} & -1 & 0 & 0 & -1 \\
\hline
Palindromic  \phantom{\Large{(}} & 0 & 1 & 1 & 0 \\
\hline
Anti-palindromic  \phantom{\Large{(}} & 0 & -1 & -1 & 0 \\
\hline
Alternating (even)  \phantom{\Large{(}} & -1 & 0 & 0 & 1 \\
\hline
Alternating (odd)  \phantom{\Large{(}} & 1 &  0 & 0 & -1 \\
\hline
\end{tabular}
\label{table:1}
\end{table}
\end{remark}

A key feature of the $\mathbf{M}_A$-structure introduced in Definition \ref{def:Mstructure} is that it is preserved under $\star$-congruence, as it is stated in the following proposition.
The proof of this result is straightforward, so it is omitted.
\begin{proposition}\label{prop:congruence}
Let $P(\lambda)\in\mathbb{F}[\lambda]^{n\times n}$, let $A\in{\rm GL}(2,\mathbb{F})$, and let $X\in\mathbb{F}^{n\times n}$ be a constant nonsingular matrix.
Then, the matrix polynomial $P(\lambda)$ is $\mathbf{M}_A$-structured if and only if the matrix polynomial $X^\star P(\lambda)X$ is $\mathbf{M}_A$-structured.
\end{proposition}

We establish in Theorem \ref{thm:singular_structure_Mobius} relationships between right and left minimal indices and bases of singular $\mathbf{M}_A$-structured matrix polynomials.
In particular, we show that the sets of right and left minimal indices of a singular $\mathbf{M}_A$-structured matrix polynomial are equal.
This theorem is a generalization of \cite[Theorems 3.4, 3.5 and 3.6]{singular} for $\mathbf{M}_A$-structured matrix polynomials.
In the proof of Theorem \ref{thm:singular_structure_Mobius}, we will use the following notation.
For any set of polynomial vectors $\mathcal{B}=\{x_1(\lambda),\hdots,x_p(\lambda) \}$ and any $A\in {\rm GL}(2,\mathbb{F})$, we denote by $\mathbf{M}_A[\mathcal{B}]$ the set $\{\mathbf{M}_A[x_1](\lambda),\hdots,\mathbf{M}_A[x_p](\lambda)  \}$, where each M\"obius transformation $\mathbf{M}_A[x_i](\lambda)$ is taken with respect to the degree of $x_i(\lambda)$, and by $\mathcal{\overline{B}}$ the set $\{  \overline{x}_1(\lambda),\hdots, \overline{x}_p(\lambda)\}$.
\begin{theorem}\label{thm:singular_structure_Mobius}
Let $A\in {\rm GL}(2,\mathbb{F})$ and let $P(\lambda)\in\mathbb{F}[\lambda]^{n\times n}$ be a singular $\mathbf{M}_A$-structured matrix polynomial.
Then, the sets of right and left minimal indices of $P(\lambda)$ are equal.
Furthermore, if $\{x_1(\lambda),\hdots,x_p(\lambda)\}$ is a  minimal basis for $\mathcal{N}_r(P)$, then $\{\mathbf{M}_{\overline{A}}[\overline{x}_1](\lambda),\hdots,\mathbf{M}_{\overline{A}}[\overline{x}_p](\lambda)  \}$ is a minimal basis for $\mathcal{N}_\ell(P)$ (modulo transposition).
\end{theorem}
\begin{proof}
Recall that left minimal indices and bases of $P(\lambda)$ can be computed as right minimal indices and bases of $P(\lambda)^T$.
Then, notice that the $\mathbf{M}_A$ structure of $P(\lambda)$ implies $P(\lambda)^T = \mathbf{M}_{\overline{A}}[\overline{P}](\lambda)$, that is, the left minimal indices of $P(\lambda)$ are equal to the right minimal indices of $\mathbf{M}_{\overline{A}}[\overline{P}](\lambda)$, and any left minimal basis of $P(\lambda)$ can be obtained as a right minimal basis of $\mathbf{M}_{\overline{A}}[\overline{P}](\lambda)$.
Clearly,  the right minimal indices of $P(\lambda)$ and $\overline{P}(\lambda)$ coincide, and if $\mathcal{B}$ is a right minimal basis for $\mathcal{N}_r(P)$, then $\mathcal{\overline{B}}$ is a right minimal basis for $\mathcal{N}_r(\overline{P})$.
Finally, from \cite[Theorem 7.5]{Mobius} together with the previous argument, we obtain that the sets of right minimal indices of $P(\lambda)$, $\overline{P}(\lambda)$ and $\mathbf{M}_{\overline{A}}[\overline{P}](\lambda)$ are equal, and that if $\mathcal{B}$ is a minimal basis for $\mathcal{N}_r(P)$, then $\mathcal{\overline{B}}$ is a basis for the right null space of $\overline{P}(\lambda)$, and, therefore, $\mathbf{M}_{\overline{A}}[\mathcal{\overline{B}}]$ is a basis for the right null space of $\mathbf{M}_{\overline{A}}[\overline{P}](\lambda)$.
\end{proof}

\medskip

Notice that the matrices in Table \ref{table:1} are coninvolutory (recall Definition \ref{def:involutory}).
For this class of matrices, we consider in the following two sections the problem of linearizing an $\mathbf{M}_A$-structured matrix polynomial of odd degree in a structure-preserving way.
To achieve this task, we need to introduce the concept of  $\mathbf{M}_A$-structured strong block minimal bases pencil, which is an important example of the recently introduced class of strong block minimal bases pencils \cite{minimal_pencils}. 

\section{Strong block minimal bases pencils and $\mathbf{M}_A$-structured strong block minimal bases pencils}\label{sec:minimal_bases_pencils}

In this section, we start reviewing the family of  strong block minimal bases pencils introduced in \cite{minimal_pencils} and, then, introduce the subfamily of $\mathbf{M}_A$-structured strong block minimal bases pencils.
\begin{definition}{\rm \cite[Definition 3.1]{minimal_pencils}} \label{def:minlinearizations} A matrix pencil
\begin{equation} \label{eq:minbaspencil}
\mathcal{L}(\lambda) = \begin{bmatrix} M(\lambda) & K_2 (\lambda)^T \\ K_1 (\lambda) & 0\end{bmatrix}
\end{equation}
is called a {\em block minimal bases pencil} if $K_1 (\lambda)$ and $K_2(\lambda)$ are both minimal bases.
If, in addition, the row degrees of $K_1 (\lambda)$ are all equal to $1$, the row degrees of $K_2 (\lambda)$ are all equal to $1$, the row degrees of a minimal basis dual to $K_1 (\lambda)$ are all equal, and the row degrees of a minimal basis dual to $K_2 (\lambda)$ are all equal, then $\mathcal{L}(\lambda)$ is called a {\em strong block minimal bases pencil}.
\end{definition}

Any (strong) block minimal bases pencil is a (strong) linearization of a certain matrix polynomial that can be expressed in terms of the pencil $M(\lambda)$ and any dual minimal bases of $K_1(\lambda)$ and $K_2(\lambda)$.
Moreover, the minimal indices of the strong block minimal bases pencil and the minimal indices of the matrix polynomial for which the pencil is a strong linearization are related by uniform shifts.
\begin{theorem}{\rm \cite[Theorems 3.3 and 3.7]{minimal_pencils}}\label{thm:blockminlin}  Let $K_1 (\lambda)$ and $N_1 (\lambda)$ be a pair of dual minimal bases, and let $K_2 (\lambda)$ and $N_2 (\lambda)$ be another pair of dual minimal bases. 
Consider the matrix polynomial
\begin{equation} \label{eq:Qpolinminbaslin}
Q(\lambda) := N_2(\lambda) M(\lambda) N_1(\lambda)^T,
\end{equation}
and the block minimal bases pencil $\mathcal{L}(\lambda)$ in \eqref{eq:minbaspencil}. 
Then:
\begin{enumerate}
\item[\rm (a)] $\mathcal{L}(\lambda)$ is a linearization of $Q(\lambda)$.
\item[\rm (b)] If $\mathcal{L}(\lambda)$ is a strong block minimal bases pencil, then $\mathcal{L}(\lambda)$ is a strong linearization of $Q(\lambda)$, considered as a polynomial with grade $1 + \deg(N_1 (\lambda)) + \deg(N_2 (\lambda))$.
\item[\rm (c1)] If $0 \leq \epsilon_1 \leq \epsilon_2 \leq \cdots \leq \epsilon_p$ are the right minimal indices of $Q(\lambda)$, then
\[
\epsilon_1 + \deg(N_1(\lambda)) \leq \epsilon_2 + \deg(N_1(\lambda)) \leq \cdots \leq \epsilon_p +  \deg(N_1(\lambda))
\]
are the right minimal indices of $\mathcal{L}(\lambda)$, when $\mathcal{L}(\lambda)$ is a strong block minimal bases pencil.
\item[\rm (c2)] If $0 \leq \eta_1 \leq \eta_2 \leq \cdots \leq \eta_q$ are the left minimal indices of $Q(\lambda)$, then
\[
\eta_1 + \deg(N_2(\lambda)) \leq \eta_2 + \deg(N_2(\lambda)) \leq \cdots \leq \eta_q +  \deg(N_2(\lambda))
\]
are the left minimal indices of $\mathcal{L}(\lambda)$, when $\mathcal{L}(\lambda)$ is a strong block minimal bases pencil.
\end{enumerate}
\end{theorem}

For any $A\in {\rm GL}(2,\mathbb{F})$, part-(b) in Theorem \ref{thm:minimal_basis_Mobius} suggests that we may take $K_2(\lambda)=\mathbf{M}_{\overline{A}}[\overline{K}_1](\lambda)$ and $N_2(\lambda)=\mathbf{M}_{\overline{A}}[\overline{N}_1](\lambda)$ as the second pair of dual minimal bases in Definition \ref{def:minlinearizations} and Theorem \ref{thm:blockminlin}.
This motivates the concept of an $\mathbf{M}_A$-structured strong block minimal bases pencil, which is introduced in the following definition.
\begin{definition}\label{def:structured-minimal-bases-pencil}
Let $K(\lambda),N(\lambda)$ be a pair of dual minimal bases, with all the row degrees of $K(\lambda)$ equal to 1 and with all the row degrees of $N(\lambda)$ equal, and let $A\in {\rm GL}(2,\mathbb{F})$ be a  coninvolutory matrix.
Then, an $\mathbf{M}_A$-structured matrix pencil of the form
\begin{equation}\label{eq:structured-minimal-bases-pencil}
\mathcal{L}(\lambda)=\begin{bmatrix}
M(\lambda) & \mathbf{M}_A[K](\lambda)^\star \\
K(\lambda) & 0 
\end{bmatrix} = \begin{bmatrix}
M(\lambda) & \mathbf{M}_{\overline{A}}[\overline{K}](\lambda)^T \\
K(\lambda) & 0
\end{bmatrix} \,\, \mbox{with} \quad \mathbf{M}_A[M](\lambda)=M(\lambda)^\star,
\end{equation}
is called a  \emph{$\mathbf{M}_A$-structured strong block minimal bases pencil}.
\end{definition}

\begin{remark}
Notice that any $\mathbf{M}_A$-structured strong block  minimal bases pencil $\mathcal{L}(\lambda)$ as in \eqref{eq:structured-minimal-bases-pencil} is, indeed, $\mathbf{M}_A$-structured, that is, $\mathbf{M}_A[\mathcal{L}](\lambda)=\mathcal{L}(\lambda)^\star$ holds as a consequence of $A$ being coninvolutory.
\end{remark}

An immediate corollary of Theorem \ref{thm:blockminlin} is that any $\mathbf{M}_A$-structured strong block minimal bases pencil  is always a strong linearization of a certain odd-grade $\mathbf{M}_A$-structured matrix polynomial.
Furthermore, the minimal indices of this polynomial and the pencil are related by a uniform shift. 
These results are stated and proved in the following theorem.
We only focus on right minimal indices, since the set of right minimal indices and the set of left minimal indices of an $\mathbf{M}_A$-structured matrix polynomial are equal (recall Theorem \ref{thm:singular_structure_Mobius}).
\begin{theorem}\label{thm:structured-minimal-basis-pencil}
Let $K(\lambda),N(\lambda)$ be a pair of dual minimal bases, with all the row degrees of $K(\lambda)$ equal to 1 and with all the row degrees of $N(\lambda)$ equal, let $A\in {\rm GL}(2,\mathbb{F})$ be a coninvolutory matrix, and let $\mathcal{L}(\lambda)$ be an $\mathbf{M}_A$-structured strong block minimal bases pencil as in \eqref{eq:structured-minimal-bases-pencil}.
Then, the pencil $\mathcal{L}(\lambda)$ is a strong linearization of the $\mathbf{M}_A$-structured matrix polynomial
\begin{equation}\label{eq:structured-matrix-equation}
Q(\lambda):=\mathbf{M}_{\overline{A}}[\overline{N}](\lambda)M(\lambda)N(\lambda)^T,
\end{equation}
of grade $2\deg(N(\lambda))+1$.
Moreover, if $0\leq \epsilon_1\leq \epsilon_2\leq \cdots \leq \epsilon_p$ are the right minimal indices of $Q(\lambda)$, then
\[
\epsilon_1+\deg(N(\lambda))\leq \epsilon_2+\deg(N(\lambda))\leq \cdots \leq \epsilon_p+\deg(N(\lambda))
\]
are the right minimal indices of $\mathcal{L}(\lambda)$.
\end{theorem}
\begin{proof} 
Notice that $K(\lambda)N(\lambda)^T=0$ implies $\overline{K}(\lambda)\overline{N}(\lambda)^T=0$.
Since the operation $P(\lambda)\rightarrow \overline{P}(\lambda)$ applied to $K(\lambda)$ and $N(\lambda)$ does not change neither the rank of the polynomial at any $\lambda_0\in\overline{\mathbb{F}}$, nor the degree of any of its entries, we have that $\overline{K}(\lambda)$ and $\overline{N}(\lambda)$ are a pair of dual minimal bases with all the row degrees of $\overline{K}(\lambda)$ equal to 1, and all the row  degrees of $\overline{N}(\lambda)$ equal to $\deg(N(\lambda))$. 
Then, from Theorem \ref{thm:minimal_basis_Mobius}, we obtain that $\mathbf{M}_{\overline{A}}[\overline{K}](\lambda)$ and $\mathbf{M}_{\overline{A}}[\overline{N}](\lambda)$ are also a pair of dual minimal bases with all the row degrees of $\mathbf{M}_{\overline{A}}[\overline{K}](\lambda)$ equal to 1, and all the row degrees of $\mathbf{M}_{\overline{A}}[\overline{N}](\lambda)$ equal to $\deg(N(\lambda))$.
Therefore, the pencil $\mathcal{L}(\lambda)$ is a strong block minimal bases pencil. 
From  Theorem \ref{thm:blockminlin}, we immediately obtain that $\mathcal{L}(\lambda)$ is a strong linearization of $Q(\lambda)$ and that the minimal indices of $\mathcal{L}(\lambda)$ are those of $Q(\lambda)$ shifted by $\deg(N(\lambda))$.

We still have to show that $Q(\lambda)$ is an $\mathbf{M}_A$-structured matrix polynomial.
Computing the M\"obius transformation of $Q(\lambda)$ associated with the matrix $A$ and using that the matrix $A$ is coninvolutory, together with parts (b), (d), (e) and (g) in Proposition \ref{prop:properties_Mobius}, we get
\begin{align*}
\mathbf{M}_A[Q](\lambda) =& \mathbf{M}_A\left[\mathbf{M}_{\overline{A}}[\overline{N}]MN^T\right](\lambda)\\  =&
 \mathbf{M}_A\left[  \mathbf{M}_{\overline{A}}[\overline{N}] \right](\lambda)\,\, \mathbf{M}_A[M](\lambda)\,\, \mathbf{M}_A[N^T](\lambda) \\ =&
 \overline{N}(\lambda)M(\lambda)^\star  \mathbf{M}_A[N](\lambda)^T =
 \left( N(\lambda)^T \right)^\star M(\lambda)^\star  \mathbf{M}_{\overline{A}}[\overline{N}](\lambda)^\star = Q(\lambda)^\star.
\end{align*}
Thus, the matrix polynomial $Q(\lambda)$ is $\mathbf{M}_A$-structured.
\end{proof}

Theorem \ref{thm:structured-minimal-basis-pencil} tells us that given an $\mathbf{M}_A$-structured strong block minimal bases pencil, this pencil is a strong linearization of a certain odd-grade $\mathbf{M}_A$-structured matrix polynomial.
We now address the inverse problem, that is, given an odd-grade $\mathbf{M}_A$-structured matrix polynomial $Q(\lambda)$, we want to construct an $\mathbf{M}_A$-structured strong linearization for the given matrix polynomial.
To achieve this, the polynomial equation \eqref{eq:structured-matrix-equation} has to be solved for an $\mathbf{M}_A$-structured pencil $M(\lambda)$.
This problem is addressed in Theorem \ref{thm:solving_M}.
But, first, we notice that if we do not impose the $\mathbf{M}_A$-structure on the pencil $M(\lambda)$, it is possible to prove that the polynomial equation \eqref{eq:structured-matrix-equation} is always consistent, with infinitely many solutions, as a consequence of the properties of the minimal basis $N(\lambda)$.
This result will be proved in \cite{ell-ifications}, in a much more general setting.
\begin{theorem}\label{thm:solving_M}
Let $K(\lambda),N(\lambda)$ be a pair of dual minimal bases, with all the row degrees of $K(\lambda)$ equal to 1 and with all the row degrees of $N(\lambda)$ equal, let $A\in {\rm GL}(2,\mathbb{F})$ be a coninvolutory matrix.
Let $Q(\lambda)$ be a given $\mathbf{M}_A$-structured matrix polynomial  with grade equal to $2\deg(N(\lambda))+1$. 
If $\widehat{M}(\lambda)$ is any solution of the polynomial equation \eqref{eq:structured-matrix-equation} (not necessarily $\mathbf{M}_A$-structured), then the pencil
\[
M(\lambda) :=\frac{1}{2}\left( \widehat{M}(\lambda)+\mathbf{M}_A[\widehat{M}](\lambda)^\star \right)
\]
is an $\mathbf{M}_A$-structured solution of \eqref{eq:structured-matrix-equation}, and the $\mathbf{M}_A$-structured strong block minimal bases pencil
\[
 \mathcal{L} (\lambda) :=
 \begin{bmatrix}
M(\lambda) & \mathbf{M}_A[K](\lambda)^\star \\
K(\lambda) & 0
\end{bmatrix} =
\begin{bmatrix}
\frac{1}{2}\left( \widehat{M}(\lambda)+\mathbf{M}_A[\widehat{M}](\lambda)^\star\right) & \mathbf{M}_A[K](\lambda)^\star \\
K(\lambda) & 0
\end{bmatrix}
\]
is an $\mathbf{M}_A$-structured strong linearization of $Q(\lambda)$.
\end{theorem}
\begin{proof}
By using that $Q(\lambda)$ is an $\mathbf{M}_A$-structured matrix polynomial, that $A$ is coninvolutory, and that $\widehat{M}(\lambda)$ is a solution of \eqref{eq:structured-matrix-equation}, together with parts (b), (d), (e) and (g) in Proposition \ref{prop:properties_Mobius}, we have
\begin{align*}
Q(\lambda)=&\left(\mathbf{M}_A[Q](\lambda)\right)^\star = \left(\mathbf{M}_A\left[ \mathbf{M}_{\overline{A}}[\overline{N}]\widehat{M}N^T \right](\lambda)\right)^\star \\=&
\left( \overline{N}(\lambda)\mathbf{M}_A[\widehat{M}](\lambda)\mathbf{M}_A[N](\lambda)^T \right)^\star  \\=
&\left( \left( N(\lambda)^T \right)^\star \mathbf{M}_A[\widehat{M}](\lambda) \mathbf{M}_{\overline{A}}[\overline{N}](\lambda)^\star  \right)^\star  \\=
&\mathbf{M}_{\overline{A}}[\overline{N}](\lambda)\, \mathbf{M}_A[\widehat{M}](\lambda)^\star \,N(\lambda)^T.
\end{align*}
Thus, the matrix pencil $\mathbf{M}_A[\widehat{M}](\lambda)^\star$ is also a solution of \eqref{eq:structured-matrix-equation}.
Moreover, since any affine combination of solutions of \eqref{eq:structured-matrix-equation} is also a solution, the $\mathbf{M}_A$-structured matrix pencil $(\widehat{M}(\lambda)+\mathbf{M}_A[\widehat{M}](\lambda)^\star)/2$ satisfies \eqref{eq:structured-matrix-equation}.
Therefore, by Theorem \ref{thm:structured-minimal-basis-pencil}, the $\mathbf{M}_A$-structured pencil $\mathcal{L}(\lambda)$ is a strong linearization of the matrix polynomial $Q(\lambda)$.
\end{proof}

Despite its consistency, the polynomial equation \eqref{eq:structured-matrix-equation} might be very difficult to solve for an arbitrary minimal basis $N(\lambda)$ and an arbitrary coninvolutory matrix $A$.
However, for some choices of $N(\lambda)$, when the matrix $A$ is any of those in Table \ref{table:1}, this problem turns out to be particularly simple. 
This is the subject of the following section.

\section{$\mathbf{M}_A$-structured block Kronecker pencils and structure-preserving strong linearizations}\label{sec:classical_structures}

We focus in this section on the problem of constructing explicitly structure-preserving strong linearizations  for (skew-)symmetric, (anti-)\break palindromic, and alternating matrix polynomials from some subfamilies of $\mathbf{M}_A$-\break structured strong block minimal bases pencils.
This problem has been addressed in \cite{PartI} with a lot of detail for (skew-)symmetric matrix polynomials and will be addressed in \cite{PartII,PartIII} for (anti-)palindromic and alternating matrix polynomials, so we only review some of the most important results and show some illuminating examples.

We start by introducing the family of $\mathbf{M}_A$-structured block Kronecker pencils, which are particular but important examples of $\mathbf{M}_A$-structured strong block minimal bases pencils.
\begin{definition}\label{def:structured-Kron-pencil}
Let $L_k(\lambda)\in\mathbb{F}[\lambda]^{k\times(k+1)}$ be the matrix pencil in \eqref{eq:Lk}, and let $A\in {\rm GL}(2,\mathbb{F})$ be a  coninvolutory matrix.
Then, a pencil of the form
\begin{equation}\label{eq:structured-Kron-pencil}
\mathcal{L}(\lambda)=\begin{bmatrix}
M(\lambda) & \mathbf{M}_A[L_k](\lambda)^\star\otimes I_n \\
L_k(\lambda)\otimes I_n & 0 
\end{bmatrix} \quad \mbox{with} \quad \mathbf{M}_A[M](\lambda)=M(\lambda)^\star,
\end{equation}
is called a \emph{$\mathbf{M}_A$-structured block Kronecker pencil}.
Moreover, the partition of $\mathcal{L}(\lambda)$ into $2\times 2$ blocks in \eqref{eq:structured-Kron-pencil} is called the \emph{natural partition} of an $\mathbf{M}_A$-structured block Kronecker pencil.
\end{definition}
\begin{remark}
The name $\mathbf{M}_A$-structured block Kronecker pencil is motivated, first, by the fact that one of the building blocks of $\mathcal{L}(\lambda)$ is the Kronecker product of a singular block of the Kronecker canonical form of pencils with the identity (as for block Kronecker pencils in \cite{minimal_pencils}), and that the pencil $\mathcal{L}(\lambda)$ is an $\mathbf{M}_A$-structured pencil, i.e. , $\mathbf{M}_A[\mathcal{L}](\lambda)=\mathcal{L}(\lambda)^\star$.
\end{remark}

As an immediate corollary of Theorem \ref{thm:structured-minimal-basis-pencil}, we obtain that any $\mathbf{M}_A$-structured block Kronecker pencil $\mathcal{L}(\lambda)$ is a strong linearization of an $\mathbf{M}_A$-structured matrix polynomial, and that the minimal indices of this polynomial and $\mathcal{L}(\lambda)$ are related by a uniform shift.
\begin{theorem}\label{thm:structured-block_Kronecker_pencils}
Let $\mathcal{L}(\lambda)$ be an $\mathbf{M}_A$-structured block Kronecker pencil  as in \eqref{eq:structured-Kron-pencil}.
Then, the pencil $\mathcal{L}(\lambda)$ is a strong linearization of the $\mathbf{M}_A$-structured matrix polynomial
\begin{equation}\label{eq:structured-matrix-equation-Kron}
P(\lambda):= \left( \mathbf{M}_{A}[\Lambda_k](\lambda)^\star\otimes I_n \right)M(\lambda)
\left( \Lambda_k(\lambda)\otimes I_n \right),
\end{equation}
of grade $2k+1$, where $\Lambda_k(\lambda)$ is the vector polynomial defined in \eqref{eq:Lambda}.
Moreover, the left minimal indices of $\mathcal{L}(\lambda)$ are those of $P(\lambda)$ increased by $k$, and the right minimal indices of $\mathcal{L}(\lambda)$ are those of $P(\lambda)$ increased also by $k$.
\end{theorem}

Following the terminology introduced in \cite{PartI,PartII,PartIII}, when the matrix $A$ is any of those listed in Table \ref{table:1}, the corresponding $\mathbf{M}_A$-structured block Kronecker pencil \eqref{eq:structured-Kron-pencil} is called  \emph{(skew-)symmetric},  \emph{(anti-)palindromic}, or \emph{alternating block Kronecker pencil}, depending on the case.
Moreover, the union of the sets of (skew-)symmetric,  (anti-)palindromic, and alternating block Kronecker pencils is called the family of \emph{structured block Kronecker pencils}.
We list in Table \ref{table:2} the minimal bases $\mathbf{M}_A[L_k](\lambda)\otimes I_n$ and the conditions on the pencil $M(\lambda)$ for structured block Kronecker pencils.
 \begin{table}[ht]
\caption{The  minimal bases $\mathbf{M}_A[L_k](\lambda)\otimes I_n$ and the conditions on $M(\lambda)=\lambda M_1+M_0$ for structured block Kronecker pencils.}
\centering 
\begin{tabular}{|l|c|r|}
\hline \hline
\bf{structure} \phantom{\Large{(}} & condition on $M(\lambda)=\lambda M_1+M_0$ & $\mathbf{M}_A[L_k](\lambda)$  \\
\hline\hline
Symmetric  \phantom{\Large{(}}  & $\lambda M_1+M_0$ with $M_0^\star=M_0$ and $M_1^\star=M_1$ & $L_k(\lambda)$ \\ 
\hline
Skew-symmetric  \phantom{\Large{(}} & $\lambda M_1+M_0$ with $M_0^\star=-M_0$ and $M_1^\star=-M_1$ & $-L_k(\lambda)$ \\
\hline
Palindromic \phantom{\Large{(}} & $\lambda M_1+M_1^\star$ & $\rev L_k(\lambda)$\\
\hline
Anti-palindromic \phantom{\Large{(}} & $\lambda M_1-M_1^\star$ & $-\rev L_k(\lambda)$  \\
\hline
Alternating (even) \phantom{\Large{(}} & $\lambda M_1+M_0$ with $M_0^\star=M_0$ and $M_1^\star=-M_1$  & $L_k(-\lambda)$ \\
\hline
Alternating (odd)  \phantom{\Large{(}} & $\lambda M_1+M_0$ with $M_0^\star=-M_0$ and $M_1^\star=M_1$ & $-L_k(-\lambda)$ \\
\hline
\end{tabular}
\label{table:2}
\end{table}

We know from Theorem \ref{thm:solving_M} that one can always construct a structure-preserving strong linearization of any $\mathbf{M}_A$-structured matrix polynomial with odd grade $g$ via an $\mathbf{M}_A$-structured block Kronecker pencil as in \eqref{eq:structured-Kron-pencil} with $k=(g-1)/2$.
Furthermore, for the (skew-)symmetric, (anti-)palindromic or alternating structures, this construction turns out to be rather simple.
In Theorem \ref{thm:givenP_findM}, we show what conditions on $M(\lambda)$ are needed for a structured block Kronecker pencil to be a structure-preserving strong linearization of a given odd-grade structured matrix polynomial.
\begin{theorem}\label{thm:givenP_findM}
Let $P(\lambda)=\sum_{i=0}^g P_i\lambda^i\in\mathbb{F}[\lambda]^{n\times n}$ be an odd-grade structured matrix polynomial, let $\mathscr{S}(P)$ be the structure of $P(\lambda)$, and let $A$ be one of the matrices in Table \ref{table:1}, depending on $\mathscr{S}(P)$.
Additionally, let $M(\lambda)=\lambda M_1+M_0\in\mathbb{F}[\lambda]^{(k+1)n\times (k+1)n}$, with $k=(g-1)/2$, be a matrix pencil, and let us partition the matrices $M_0$ and $M_1$ into $(k+1)\times(k+1)$ blocks each of size $n\times n$ and let us denote these blocks by $[M_1]_{ij},[M_0]_{ij}\in\mathbb{F}^{n\times n}$ for $i,j=1,2,\hdots,k+1$.  
If the following condition holds, for $\ell=0,1,\hdots,g$, 
\begin{equation}\label{eq:condition_sym}
P_\ell = \sum_{i+j=g+2-\ell}[M_1]_{ij}+\sum_{i+j=g+1-\ell}[M_0]_{ij},
\end{equation}
when $\mathscr{S}(P)\in\{\mbox{symmetric, skew-symmetric}\}$, or
\begin{equation}\label{eq:condition_pal}
P_\ell = \sum_{i-j=\ell-k-1}[M_1]_{ij}+\sum_{i-j=\ell-k}[M_0]_{ij},
\end{equation}
 when $\mathscr{S}(P)\in\{\mbox{palindromic, anti-palindromic}\}$, or
\begin{equation}\label{eq:condition_alt}
P_\ell = \sum_{i+j=g+2-\ell}(-1)^{k-i+1}[M_1]_{ij}+\sum_{i+j=g+1-\ell}(-1)^{k-i+1}[M_0]_{ij},
\end{equation}
 when $\mathscr{S}(P)\in\{\mbox{even, odd}\}$,  then the matrix pencil
\[
\mathcal{L}(\lambda) = \begin{bmatrix}
\frac{1}{2}(M(\lambda)+\mathbf{M}_A[M](\lambda)^\star) & \mathbf{M}_A[L_k](\lambda)^\star\otimes I_n \\
L_k(\lambda)\otimes I_n & 0
\end{bmatrix}
\]
is an $\mathbf{M}_A$-structured block Kronecker pencil such that:

\smallskip

\begin{enumerate}
\item[\rm (i)] $\mathcal{L}(\lambda)$ is a strong linearization of $P(\lambda)$,
\item[\rm (ii)] $\mathcal{L}(\lambda)$ and $P(\lambda)$ share the same structure, i.e., $\mathscr{S}(P)=\mathscr{S}(\mathcal{L})$, and
\item[\rm (iii)] the  left minimal indices of $\mathcal{L}(\lambda)$ are those of $P(\lambda)$ increased by $k$, and the right minimal indices of $\mathcal{L}(\lambda)$ are those of $P(\lambda)$ increased by $k$. 
\end{enumerate}

\smallskip

\end{theorem}
\begin{proof}
Clearly, the structured block Kronecker pencil $\mathcal{L}(\lambda)$ and the matrix polynomial $P(\lambda)$ share the same structure, that is, part (ii) holds.
To prove parts (i) and (iii), we just need to check that \eqref{eq:structured-matrix-equation-Kron} holds for $(M(\lambda)+\mathbf{M}_A[M](\lambda)^\star)/2$ (up to a sign), since the desired results would follow from Theorem  \ref{thm:structured-block_Kronecker_pencils}, together with the fact that any strong linearization of $-P(\lambda)$ is also a strong linearization of $P(\lambda)$, and the sets of right and left minimal indices of $P(\lambda)$ and $-P(\lambda)$ are the same.

Several cases have to be distinguished. 
For brevity, we focus only on the case $\mathscr{S}(P)\in\{symmetric,skew-symmetric\}$.
The proofs for the other cases are very similar, so we invite the reader to complete the proof.
First, assume that $\mathscr{S}(P)\in\{symmetric\}$.
Then, we have $\mathbf{M}_A[\Lambda_k](\lambda)\otimes I_n = \Lambda_k(\lambda)\otimes I_n$ and $\mathbf{M}_A[M](\lambda)^\star = M(\lambda)^\star $, so, in this case, the structured block Kronecker pencil $\mathcal{L}(\lambda)$ is a strong linearization of $Q_1(\lambda):=(\Lambda_k(\lambda)^T\otimes I_n)(M(\lambda)+M(\lambda)^\star)(\Lambda_k(\lambda)\otimes I_n)/2$.
A direct multiplication, some basic manipulations and the condition \eqref{eq:condition_sym} yield $P(\lambda)=Q_1(\lambda)$.
Therefore, the result is true in this  case.
Now, assume $\mathscr{S}(P)\in\{skew-symmetric\}$.
In this  case, we have $\mathbf{M}_A[\Lambda_k](\lambda)\otimes I_n =(-1)^k \Lambda_k(\lambda)\otimes I_n$ and $\mathbf{M}_A[M](\lambda)^\star =- M(\lambda)^\star $.
Therefore, $\mathcal{L}(\lambda)$ is a strong linearization of $Q_2(\lambda):=(-1)^k(\Lambda_k(\lambda)^T\otimes I_n)(M(\lambda)-M(\lambda)^\star)(\Lambda_k(\lambda)\otimes I_n)/2$.
It is not difficult to check that the condition \eqref{eq:condition_sym} implies that $Q_2(\lambda)=(-1)^k P(\lambda)$. 
Thus, the result is also true in this case.
\end{proof}

We illustrate Theorem \ref{thm:givenP_findM} in Examples \ref{ex:1}, \ref{ex:2} and \ref{ex:3}, where we construct structure-preserving strong linearizations for grade-7 symmetric, palindromic and even
 matrix polynomials, respectively.
\begin{example}\label{ex:1}
Let $P(\lambda) = \sum_{i=0}^7P_i\lambda^i\in\mathbb{F}[\lambda]^{n\times n}$ be a symmetric grade-7 matrix polynomial, and consider the following matrix pencil
\[
M_1(\lambda):=
\begin{bmatrix}
\lambda P_7 & 0 & 0 & 0 \\
\lambda P_6 +P_5 & P_4 & P_3 &0 \\
0 & 0 & P_2 & 0 \\
0 & 0 & P_1 & P_0
\end{bmatrix}.
\]		
It is easy to check that the pencil $M_1(\lambda)$ satisfies \eqref{eq:condition_sym} with $g=7$.
In this case, we have $A=\left[\begin{smallmatrix} 1 & 0 \\ 0 & 1 \end{smallmatrix}\right]$, so the pencil $(M_1(\lambda)+\mathbf{M}_A[M_1](\lambda)^\star)/2$ is given by
\[
\frac{1}{2}(M_1(\lambda)+M_1(\lambda)^\star)=
\begin{bmatrix}
\lambda P_7 & (\lambda P_6+P_5)/2 & 0 & 0 \\
(\lambda P_6+P_5)/2 & P_4 & P_3/2 & 0 \\
0 & P_3/2 & P_2 & P_1/2 \\
0 & 0 & P_1/2 & P_0
\end{bmatrix},
\]
which is a symmetric pencil.
We conclude, by Theorem \ref{thm:givenP_findM}, that the symmetric block Kronecker pencil
\[
\left[\begin{array}{cccc|ccc}
\lambda P_7 & (\lambda P_6+P_5)/2 & 0 & 0 & -I_n & 0 & 0\\
(\lambda P_6+P_5)/2 & P_4 & P_3/2 & 0 & \lambda I_n & -I_n & 0\\
0 & P_3/2 & P_2 & P_1/2 & 0 & \lambda I_n & -I_n \\
0 & 0 & P_1/2 & P_0 & 0 & 0 & \lambda I_n \\ \hline
-I_n & \lambda I_n & 0 & 0 & 0 & 0 & 0\\
0 & -I_n & \lambda I_n & 0 & 0 & 0 & 0\\
0 & 0 & -I_n & \lambda I_n & 0 & 0 & 0
\end{array}\right]
\]
is a symmetric strong linearization of $P(\lambda)$.
\end{example}
\begin{example}\label{ex:2}
Let $P(\lambda) = \sum_{i=0}^7P_i\lambda^i\in\mathbb{F}[\lambda]^{n\times n}$ be a palindromic grade-7 matrix polynomial, and consider the following matrix pencil
\[
M_2(\lambda):=
\begin{bmatrix}
0 & 0 & P_1 & P_0 \\
0 & P_3 & P_2 & 0 \\
\lambda P_6 & \lambda P_5+P_4 & 0 & 0 \\
\lambda P_7 & 0 & 0 & 0
\end{bmatrix}.
\]		
It is easy to check that the pencil $M_2(\lambda)$ satisfies \eqref{eq:condition_pal} with $g=7$.
For the palindromic structure we have $A=\left[\begin{smallmatrix} 0 & 1 \\ 1 & 0 \end{smallmatrix}\right]$, so the pencil $(M_2(\lambda)+\mathbf{M}_A[M_2](\lambda)^\star)/2$ is given by
\[
\frac{1}{2}(M_2(\lambda)+\rev_1 M_2(\lambda)^\star)=
\begin{bmatrix}
0 & 0 & P_1 & P_0 \\
0 & (\lambda P_4+P_3)/2 & \lambda P_3/2+P_2 & 0 \\
\lambda P_6 & \lambda P_5+P_4/2 & 0 & 0 \\
\lambda P_7 & 0 & 0 & 0
\end{bmatrix},
\]
which is a palindromic pencil.
Then, from Theorem \ref{thm:givenP_findM}, we obtain  that the palindromic block Kronecker pencil
\[
\left[\begin{array}{cccc|ccc}
0 & 0 & P_1 & P_0 & -\lambda I_n & 0 & 0\\
0 & (\lambda P_4+P_3)/2 & \lambda P_3/2+P_2 & 0 & I_n & -\lambda I_n & 0\\
\lambda P_6 & \lambda P_5+P_4/2 & 0 & 0 & 0 & I_n &-\lambda I_n \\
\lambda P_7 & 0 & 0 & 0 & 0 & 0 &  I_n \\ \hline
-I_n & \lambda I_n & 0 & 0 & 0 & 0 & 0\\
0 & -I_n & \lambda I_n & 0 & 0 & 0 & 0\\
0 & 0 & -I_n & \lambda I_n & 0 & 0 & 0
\end{array}\right]
\]
is a palindromic strong linearization of $P(\lambda)$.
\end{example}

\begin{example}\label{ex:3}
Let $P(\lambda) = \sum_{i=0}^7P_i\lambda^i\in\mathbb{F}[\lambda]^{n\times n}$ be an even grade-7 matrix polynomial, and consider the following matrix pencil
\[
M_3(\lambda):=
\begin{bmatrix}
-\lambda P_7 & 0 & 0 & 0 \\
\lambda P_6 & \lambda P_5+P_4 & P_3 & 0 \\
0 & 0 & -P_2 & 0 \\
0 & 0 & 0 & \lambda P_1+P_0
\end{bmatrix}.
\]		
It is easy to check that the pencil $M_3(\lambda)$ satisfies \eqref{eq:condition_alt} with $g=7$.
For even-structured matrix polynomials, we have $A=\left[\begin{smallmatrix} -1 & 0 \\ 0 & 1 \end{smallmatrix}\right]$, so the pencil $(M_3(\lambda)+\mathbf{M}_A[M_3](\lambda)^\star)/2$ is given by
\[
\frac{1}{2}(M_3(\lambda)+ M_3(-\lambda)^\star)=
\begin{bmatrix}
-\lambda P_7 & -\lambda P_6/2 & 0 & 0 \\
\lambda P_6/2 & \lambda P_5+P_4 & P_3/2 & 0 \\
0 & -P_3/2 & -P_2 & 0 \\
0 & 0 & 0 & \lambda P_1+P_0
\end{bmatrix},
\]
which is an even pencil.
We conclude, by Theorem \ref{thm:givenP_findM}, that the even block Kronecker pencil
\[
\left[\begin{array}{cccc|ccc}
-\lambda P_7 & -\lambda P_6/2 & 0 & 0  & -I_n & 0 & 0\\
\lambda P_6/2 & \lambda P_5+P_4 & P_3/2 & 0 & -\lambda I_n & -I_n & 0\\
0 & -P_3/2 & -P_2 & 0 & 0 & -\lambda I_n & -I_n \\
0 & 0 & 0 & \lambda P_1+P_0 & 0 & 0 & -\lambda I_n \\ \hline
-I_n & \lambda I_n & 0 & 0 & 0 & 0 & 0\\
0 & -I_n & \lambda I_n & 0 & 0 & 0 & 0\\
0 & 0 & -I_n & \lambda I_n & 0 & 0 & 0
\end{array}\right]
\]
is an even strong linearization of $P(\lambda)$.
\end{example}

We now focus on the famous block-tridiagonal and block-antitridiagonal structure-preserving linearizations introduced in  \cite{Greeks2,Alternating,Palindromic,Skew}.
We show in Example \ref{ex:tridiagonal}, for a small-grade case, that (modulo permutations) they are structured block Kronecker pencils.
 The extension of this result to any odd-grade matrix polynomial is straightforward.
\begin{example}\label{ex:tridiagonal}
Let  $\sum_{i=0}^{5}P_i\lambda^i\in\mathbb{F}[\lambda]^{n\times n}$ be a grade-5 structured matrix polynomial, and let $\sigma\in\{-1,1\}$.
Consider, first, the block-tridiagonal pencil
\[
\mathcal{L}_1(\lambda)=\begin{bmatrix}
\lambda P_5+P_4 & -\sigma I_n & 0 & 0 & 0 \\
-I_n & 0 & \lambda I_n & 0 & 0 \\
0 & \sigma \lambda I_n & \lambda P_3+P_2 & -\sigma I_n & 0 \\
0 & 0 & -I_n & 0 & \lambda I_n \\
0 & 0 & 0 & \sigma \lambda I_n & \lambda P_1+P_0
\end{bmatrix}.
\]
This pencil is a strong linearization of $P(\lambda)$ that for $\sigma=1$ is symmetric when $P(\lambda)$ is, or for $\sigma=-1$ is skew-symmetric when $P(\lambda)$ is \cite{Greeks2}.
Then, consider the block-antitridiagonal pencil
\[
\mathcal{L}_2(\lambda)=\begin{bmatrix}
0 & 0 & 0 & -\sigma \lambda I_n & \lambda P_1+P_0 \\
0 & 0 & -I_n & 0 & \lambda I_n \\
0 & -\sigma \lambda I_n & \lambda P_3+P_2 & \sigma I_n & 0 \\
-I_n & 0 & \lambda I_n & 0 & 0 \\
\lambda P_5+P_4 & \sigma I_n & 0 & 0 & 0
\end{bmatrix}.
\]
The above pencil is a strong linearization of $P(\lambda)$ that for $\sigma=1$ is palindromic when $P(\lambda)$ is, or for $\sigma=-1$ is anti-palindromic when $P(\lambda)$ is \cite{Palindromic}.
Finally, consider the block-tridiagonal pencil
\[
\mathcal{L}_3(\lambda)=\begin{bmatrix}
\lambda P_5+P_4 & -\sigma I_n & 0 & 0 & 0 \\
-I_n & 0 & \lambda I_n & 0 & 0 \\
0 & -\sigma \lambda I_n & -\lambda P_3-P_2 & -\sigma I_n & 0 \\
0 & 0 & -I_n & 0 & \lambda I_n \\
0 & 0 & 0 & -\sigma \lambda I_n & \lambda P_1+P_0
\end{bmatrix}.
\]
This pencil is a strong linearization of $P(\lambda)$ that for $\sigma=1$ is even when $P(\lambda)$ is, or for $\sigma=-1$ is odd when $P(\lambda)$ is \cite{Alternating}.
 Moreover, it is not difficult to show that there exist permutation matrices $\Pi_1,\Pi_2,\Pi_3$ such that
 \begin{align*}
 &\Pi_1\mathcal{L}_1(\lambda)\Pi_1^\star =
 \left[\begin{array}{ccc|cc}
 \lambda P_5+P_4 & 0 & 0 & -\sigma I_n & 0 \\
 0 & \lambda P_3+P_2 & 0 & \sigma \lambda I_n & -\sigma I_n \\
 0 & 0 & \lambda P_1+P_0 & 0 & \sigma \lambda I_n \\ \hline
 -I_n & \lambda I_n & 0 & 0 & 0 \\
 0 & -I_n & \lambda I_n & 0 & 0
 \end{array}\right], \\
  &\Pi_2\mathcal{L}_2(\lambda)]\Pi_2^\star =
 \left[\begin{array}{ccc|cc}
 0 & 0 & \lambda P_1+P_0 & -\sigma \lambda I_n & 0 \\
 0 & \lambda P_3+P_2 & 0 & \sigma I_n & -\sigma \lambda I_n \\
 \lambda P_5+P_4 & 0 & 0 & 0 & \sigma I_n \\ \hline
  -I_n & \lambda I_n & 0 & 0 & 0 \\
 0 & -I_n & \lambda I_n & 0 & 0
  \end{array}\right], \quad \mbox{and}\\
   &\Pi_3\mathcal{L}_3(\lambda)\Pi_3^\star =
 \left[\begin{array}{ccc|cc}
 \lambda P_5+P_4 & 0 & 0 & -\sigma I_n & 0 \\
 0 & -\lambda P_3-P_2 & 0 & -\sigma \lambda I_n & -\sigma I_n \\
 0 & 0 & \lambda P_1+P_0 & 0 & -\sigma \lambda I_n \\ \hline
 -I_n & \lambda I_n & 0 & 0 & 0 \\
 0 & -I_n & \lambda I_n & 0 & 0
 \end{array}\right].
 \end{align*}
 In other words, the block-tridiagonal and block-antitridiagonal structure-preserving linearizations are (up to a permutation) structured block Kronecker pencils.
\end{example}

\medskip

The next section is devoted to the backward error analysis when the structured complete polynomial eigenvalue problem is solved via a structure-preserving linearization obtained from a structured block Kronecker pencil, and a structurally global backward stable generalized eigensolver.

\section{Global and structured backward error analysis}\label{sec:analysis}

As we mentioned in the introduction, the structured complete polynomial eigenvalue problem consists of computing all the eigenvalues, finite and infinite, and all the minimal indices, left and right, of a structured matrix polynomial $P(\lambda)$ using an algorithm that preserves the spectral symmetries of $P(\lambda)$ in a floating point arithmetic environment.
For example, for palindromic or alternating matrix polynomials, the structured version of the staircase algorithm for pencils developed in \cite{Schroder_thesis} can be applied to any structure-preserving strong linearization of the matrix polynomial whose minimal indices are related to those of $P(\lambda)$ via simple rules.
When the palindromic matrix polynomial is regular, the problem consists just of computing finite and infinite eigenvalues.
In this case, the preferred method is the palindromic-QR algorithm \cite{Implicit_palQR,palQR}.

Some of the structure-preserving generalized eigensolvers, such as the structured version of the staircase algorithm mentioned above and the palindromic-QR algorithm, are structurally backward stable. 
This means, that if they are applied to any structure-preserving strong linearization $\mathcal{L}(\lambda)$ of a structured matrix polynomial $P(\lambda)$ in a computer with unit roundoff $\mathbf{u}$, then the computed complete eigenstructure of $\mathcal{L}(\lambda)$ is the exact complete eigenstructure of a matrix pencil $\mathcal{L}(\lambda)+\Delta \mathcal{L}(\lambda)$ such that 
\begin{equation}\label{eq:pencil_backward_stable}
\frac{\| \Delta \mathcal{L}(\lambda) \|_F}{\|\mathcal{L}(\lambda)\|_F} = O(\bf{u}) \quad  \mbox{and} \quad \mathscr{S}(\Delta \mathcal{L}) = \mathscr{S}(\mathcal{L}),
\end{equation}
which for the structures considered in this work, that is, (skew-)symmetric, (anti-)\break palindromic and alternating, is equivalent to $\mathscr{S}(\mathcal{L}+\Delta \mathcal{L})= \mathscr{S}(\mathcal{L})$.
However, it is not obvious whether or not \eqref{eq:pencil_backward_stable} guarantees that the computed complete eigenstructure of $P(\lambda)$ is the exact complete eigenstructure of a nearby matrix polynomial $P(\lambda)+\Delta P(\lambda)$ of the same grade as $P(\lambda)$ such that 
\begin{equation}\label{eq:poly_backward_stable}
\frac{\|\Delta P(\lambda)\|_F}{\|P(\lambda)\|_F} = O(\mathbf{u}) \quad \mbox{and} \quad \mathscr{S}(\Delta P) = \mathscr{S}(P)\left(= \mathscr{S}(\mathcal{L})\right).
\end{equation}

The goal of this section is to study this question for the family of structured block Kronecker pencils, and its answer can be found in Theorem \ref{thm:back_error_main_theorem} and Corollary \ref{cor:FINperturbation}.
Before proceeding, we remark that our structured backward error analysis follows  closely the unstructured analysis in the recent work \cite[Section 6]{minimal_pencils}. 
However, there are some very important differences in our analysis that we will highlight.
Also, to help the reader to follow the argument that leads to Theorem \ref{thm:back_error_main_theorem}, we start by sketching the main ideas and the three steps in which the backward error analysis is split.

\medskip

\noindent {\bf Initial data}. A structured ((skew-)symmetric, (anti-)palindromic or alternating) matrix polynomial $P(\lambda)=\sum_{i=0}^g P_i\lambda^i\in\mathbb{F}[\lambda]^{n\times n}$ and a structured block Kronecker pencil $\mathcal{L}(\lambda)$ as in \eqref{eq:structured-Kron-pencil}, where $A$ is one of the matrices in Table \ref{table:1} and it is chosen to guarantee $\mathscr{S}(P)=\mathscr{S}(\mathcal{L})$, such that 
\begin{equation}\label{eq:poly_section6}
P(\lambda) = \left( \mathbf{M}_{A}[\Lambda_k](\lambda)^\star\otimes I_n\right)M(\lambda)\left( \Lambda_k(\lambda)\otimes I_n \right), \quad \mbox{with }2k+1=g,
\end{equation}
are given. 
A perturbation $\Delta \mathcal{L}(\lambda)$ of the pencil $\mathcal{L}(\lambda)$ with $\mathscr{S}(\Delta \mathcal{L})=\mathscr{S}(\mathcal{L})$ is also given.
We will partition the perturbed pencil $\mathcal{L}(\lambda)+\Delta \mathcal{L}(\lambda)$ into blocks conformable to those of the natural partition of $\mathcal{L}(\lambda)$, that is,
\begin{equation}\label{eq:partition_perturbed_pencil}
\mathcal{L}(\lambda)+\Delta \mathcal{L}(\lambda) = 
\left[
\begin{array}{c|c}
\lambda M_1+M_0 + \Delta \mathcal{L}_{11}(\lambda) & \mathbf{M}_A[L_k](\lambda)^\star \otimes I_n + \mathbf{M}_A[\Delta \mathcal{L}_{21}](\lambda)^\star \\ \hline
L_k(\lambda)\otimes I_n + \Delta\mathcal{L}_{21}(\lambda) & \Delta\mathcal{L}_{22}(\lambda) 
\end{array}
\right],
\end{equation}
where the relation between the blocks $(1,2)$ and $(2,1)$ of the pencil $\Delta\mathcal{L}(\lambda)$ is forced by $\mathbf{M}_A[\Delta\mathcal{L}](\lambda)=\Delta\mathcal{L}(\lambda)^\star$ and $\lambda M_1+M_0:=M(\lambda)$.

\smallskip

\noindent {\bf First step}.
We establish a bound on $\|\Delta \mathcal{L}(\lambda)\|_F$ that allows us to construct an $\star$-congruence transformation that puts the (2,2)-block of the perturbed pencil back to zero, preserving simultaneously the structure of the pencil (recall Proposition \ref{prop:congruence}):
\begin{align}\label{eq:back_to_zero}
\begin{split}
&\begin{bmatrix}
I_{(k+1)n} & 0 \\
X & I_{kn}
\end{bmatrix}
\left( \mathcal{L}(\lambda)+\Delta \mathcal{L}(\lambda) \right)
\begin{bmatrix}
I_{(k+1)n} & X^\star \\
0 & I_{kn}
\end{bmatrix}  \\=
&\left[\begin{array}{c|c}
\lambda M_1+M_0 + \Delta \mathcal{L}_{11}(\lambda) & \mathbf{M}_A[L_k](\lambda)^\star \otimes I_n + \mathbf{M}_A[\Delta\mathcal{\widetilde{L}}_{21}](\lambda)^\star \\ \hline
L_k(\lambda)\otimes I_n + \Delta\mathcal{\widetilde{L}}_{21}(\lambda) & \phantom{\Big{(}} 0 \phantom{\Big{(}}
\end{array}
\right]\\=:
&\mathcal{L}(\lambda)+\Delta \mathcal{\widetilde{L}}(\lambda).
\end{split}
\end{align} 

The construction in \eqref{eq:back_to_zero} is equivalent to solving a nonlinear system of $\star$-Sylvester-like equations  whose unknown is the  matrix $X$. 
Further, we obtain detailed bounds on $\|X\|_F$ and $\|\Delta \mathcal{\widetilde{L}}_{21}(\lambda)\|_F$ in terms of $\| \Delta \mathcal{L}(\lambda)\|_F$. 
It is important to remark that the pencils $\mathcal{L}(\lambda)+\Delta \mathcal{L}(\lambda)$ and $\mathcal{L}(\lambda)+\Delta \mathcal{\widetilde{L}}(\lambda)$ have the same complete eigenstructure, since  $\star$-congruence transformations are strict equivalence transformations.
We emphasize that the key reason to use an $\star$-congruence transformation, instead of the strict equivalence transformation  in \cite [Section 6]{minimal_pencils}, is that the structure of the pencil is preserved under $\star$-congruence, that is, $\mathscr{S}(\Delta \widetilde{L})=\mathscr{S}(\Delta L)$.

\smallskip

\noindent {\bf Second step}.
By using the main results in \cite[Section 6.2]{minimal_pencils}, we obtain bounds on $\| \Delta \mathcal{\widetilde{L}}_{21}(\lambda) \|_F$ that guarantee that $\mathcal{L}(\lambda)+\Delta \mathcal{\widetilde{L}}(\lambda)$ in \eqref{eq:back_to_zero} is an $\mathbf{M}_A$-structured strong block minimal bases pencil. 
As the second step in the analysis in \cite[Section 6]{minimal_pencils}, this requires two sub-steps: (i) to prove that $K(\lambda) = L_k(\lambda)\otimes I_n+\Delta \mathcal{\widetilde{L}}_{21}(\lambda)$ is a  minimal basis with all its row degrees equal to 1, and (ii) to show that there exists a minimal basis
\[
N(\lambda) = \Lambda_k(\lambda)^T\otimes I_n + \Delta R_k(\lambda)^T
\]
dual to $K(\lambda)$ with all its row degrees equal to $k$. 
Notice that the sub-steps (i) and (ii), together with Theorem \ref{thm:minimal_basis_Mobius}, imply that $\mathbf{M}_A[K](\lambda)$ and $\mathbf{M}_A[N](\lambda)$ are dual minimal bases with all its row degrees equal, respectively, to 1 and $k$. 

\smallskip

\noindent {\bf Third step}. The results in the first and second steps, together with Theorem \ref{thm:structured-minimal-basis-pencil}, imply that the $\mathbf{M}_A$-structured pencil $\mathcal{L}(\lambda)+\Delta \mathcal{L}(\lambda)$ in \eqref{eq:partition_perturbed_pencil} is a strong linearization of the $\mathbf{M}_A$-structured  matrix polynomial 
\begin{align}
\label{eq:poly_perturbed}
&P(\lambda) + \Delta P(\lambda) \\ &:= 
\nonumber
\left( \mathbf{M}_{A}[\Lambda_k](\lambda)^\star\otimes I_n+\mathbf{M}_{A}[\Delta R_k](\lambda)^\star \right)
( M(\lambda) + \Delta \mathcal{L}_{11}(\lambda) )
\left( \Lambda_k(\lambda)\otimes I_n+\Delta R_k(\lambda) \right),
\end{align}
and that the right and left minimal indices of  $\mathcal{L}(\lambda)+\Delta \mathcal{L}(\lambda)$ are those of $P(\lambda)+\Delta P(\lambda)$ shifted by $k$. 
Then, combining the bounds obtained in the first and second steps, we obtain a bound on $\|\Delta P(\lambda)\|_F/\| P(\lambda)\|_F$ in terms of $\|\Delta \mathcal{L}(\lambda) \|_F / \| \mathcal{L}(\lambda) \|_F$. 
Finally, the consequences of this bound are discussed.

\medskip

In the following three subsections we develop in detail the three steps that we have outlined above. 
One final remark before continuing is that, since the matrices in Table \ref{table:1} are all real, we will use without saying it explicitly that $\overline{A}=A$ (except in Theorem \ref{thm:from_T-Syl_to_Syl}, which is true for any coninvolutory matrix).

\subsection{First step: solving a system of quadratic $\star$-Sylvester-like matrix equations for constructing the $\star$-congruence}\label{sec:first_step}

Here and thereafter, we use a notation similar to the notation introduced in \cite[Section 6.1]{minimal_pencils} for the (2,1)-block of a structured block Kronecker pencil \eqref{eq:structured-Kron-pencil}, this is, $L_k(\lambda)\otimes I_n =: \lambda F_k\otimes I_n - E_k\otimes I_n =: \lambda F_{kn}-E_{kn}$, where 
\begin{equation}\label{eq:bases_definitions}
  E_{kn}=
  \left[
    \begin{array}{cc}
      I_{k}&0_{k\times 1}
    \end{array}
  \right]\otimes I_n\>,\quad \mbox{and} \quad
  F_{kn}=
  \left[
    \begin{array}{cc}
      0_{k\times 1}&I_{k}
    \end{array}
  \right]\otimes I_n\>.
\end{equation}
In addition, the natural blocks of the perturbation $\Delta \mathcal{L}(\lambda)$ in \eqref{eq:partition_perturbed_pencil} are denoted by
\begin{equation} \label{eq:blocksofdeltaL}
 \Delta \mathcal{L}(\lambda) =:
\left[\begin{array}{c|c}
\lambda \Delta B_{11}+\Delta A_{11} & \lambda (a\Delta B_{21}+c\Delta A_{21})^\star+ (b \Delta B_{21}+d\Delta A_{21})^\star \\ \hline
\lambda \Delta B_{21}+\Delta A_{21} & \lambda \Delta B_{22}+\Delta A_{22}
\end{array}\right] \, ,
\end{equation}
where recall that $A=\left[\begin{smallmatrix} a & b\\ c & d \end{smallmatrix} \right]$ is the matrix defining the M\"obius transformation $\mathbf{M}_A$, and we introduce the following two matrices
\begin{equation}\label{eq:EFhat}
\widehat{F}_{kn}:=F_{kn}+\Delta B_{21}, \quad \mbox{and} \quad
\widehat{E}_{kn}:=-E_{kn}+\Delta A_{21}.
\end{equation}

We start with the simple Lemma \ref{lem:trivial}, where we show that the construction of the $\star$-congruence in \eqref{eq:back_to_zero} is equivalent to solve a system of nonlinear $\star$-Sylvester-like equations.
The proof is a direct algebraic manipulation and is omitted.
\begin{lemma} \label{lem:trivial} There exists a constant matrix $X\in\FF^{kn \times (k+1)n}$ satisfying \eqref{eq:back_to_zero} if and only if
\begin{equation}\label{eq:gen_sylv2}
   \left[
    \begin{array}{cc}
        X&I_{kn}
    \end{array}
  \right]
  (\mathcal{L}(\lambda)+\Delta \mathcal{L}(\lambda))
  \left[
    \begin{array}{c}
      X^\star\\
      I_{kn}
    \end{array}
  \right]
  =0\>.
\end{equation}
Moreover, with the notation introduced in \eqref{eq:bases_definitions}, \eqref{eq:blocksofdeltaL} and \eqref{eq:EFhat}, the equation \eqref{eq:gen_sylv2} is equivalent to the following system of quadratic $\star$-Sylvester-like matrix equations
\begin{equation}\label{eq:gen_T-sylv}
\left\{ \begin{array}{l}
      X (b\widehat{F}_{kn}+d\widehat{E}_{kn} )^\star+\widehat{E}_{kn} X^\star= -\Delta A_{22}-f_{M_0+\Delta A_{11}}(X)\\
      X (a\widehat{F}_{kn}+c\widehat{E}_{kn})^\star+\widehat{F}_{kn}X^\star= -\Delta B_{22}-f_{M_1+\Delta B_{11}}(X)
    \end{array}\right.,
\end{equation}
for the unknown matrix $X$, where $f_M(X)$ is the following quadratic matrix function 
\begin{equation}\label{eq:f1_f2}
f_M(X):=XMX^\star.
\end{equation}

\end{lemma}

%The system of equations \eqref{eq:gen_T-sylv} is equivalent to a system of $k^2n^2$ quadratic scalar equations in the $2k(k+1)n^2$  unknown entries of $X$.
%Since the number of unknowns is strictly larger than the number of equations, the system \eqref{eq:gen_T-sylv} is an underdetermined system of equations that may have infinitely many solutions.
Our goal is to establish conditions on $\|\Delta \mathcal{L}(\lambda)\|_F$ that guarantee the existence of a solution $X$ to \eqref{eq:gen_T-sylv} with $\|X\|_F \lesssim \|\Delta \mathcal{L}(\lambda)\|_F$.
Such a solution will be obtained in Theorem \ref{thm:gen_sylvester_solution} via the following fixed point iteration:
\begin{align}
\nonumber &\mbox{Solve for $X_0$ the system of linear $\star$-Sylvester equations:} \\ \label{eq:fixed_point_it_X0}
&\left\{\begin{array}{l}
      X_0 (b\widehat{F}_{kn}+d\widehat{E}_{kn} )^\star+\widehat{E}_{kn}X_0^\star= -\Delta A_{22}\\
      X_0 (a\widehat{F}_{kn}+c\widehat{E}_{kn})^\star+\widehat{F}_{kn}X_0^\star= -\Delta B_{22}
    \end{array}\right. \, .\\
\nonumber &\mbox{For $i\geq 1$, solve for $X_i$ the system of linear $\star$-Sylvester equations:}\\ \label{eq:fixed_point_it}
&\left\{ \begin{array}{l}
      X_i (b\widehat{F}_{kn}+d\widehat{E}_{kn} )^\star+\widehat{E}_{kn}X_i^\star= -\Delta A_{22}-f_{M_0+\Delta A_{11}}(X_{i-1})\\
      X_i (a\widehat{F}_{kn}+c\widehat{E}_{kn})^\star+\widehat{F}_{kn}X_i^\star= -\Delta B_{22}-f_{M_1+\Delta B_{11}}(X_{i-1})
    \end{array}\right.\, .
\end{align}
This fixed point iteration idea, whose origin can be traced back to the work by Stewart \cite{Stewart}, is similar to the one for proving \cite[Theorem 6.8]{minimal_pencils}.
However, we emphasize that the corresponding matrix equations are rather different.

Notice that at every step of the fixed point iteration \eqref{eq:fixed_point_it_X0}-\eqref{eq:fixed_point_it} we have to solve a system of linear $\star$-Sylvester equations.
To help us to solve those equations we present Theorem \ref{thm:from_T-Syl_to_Syl}, where we relate the solution of a kind of systems of $\star$-Sylvester equations with the solution of  certain systems of Sylvester equations. 
\begin{theorem}\label{thm:from_T-Syl_to_Syl}
Let $E,F\in \mathbb{F}^{m\times n}$, let $A=\left[\begin{smallmatrix} a & b \\ c & d \end{smallmatrix}\right]\in {\rm GL}(2,\mathbb{F})$ be a coninvolutory matrix, and let $C(\lambda)=\lambda C_1+C_0\in\mathbb{F}[\lambda]^{m\times m}$ be an $\mathbf{M}_A$-structured pencil.
Let $T_{(A,E,F)}(X)$ be the linear operator
\begin{align*}
T_{(A,E,F)}:\mathbb{F}^{m\times n}&\longrightarrow \mathbb{F}^{2m\times m} \\
X &\longrightarrow \begin{bmatrix}
T_0(X) \\ T_1(X)
\end{bmatrix} := \begin{bmatrix}
X(bF+dE)^\star+EX^\star \\
X(aF+cE)^\star+FX^\star
\end{bmatrix},
\end{align*}
and let $S_{(A,E,F)}(Y,Z)$ be the following bilinear operator
\begin{align*}
S_{(A,E,F)}:\mathbb{F}^{m\times n}\times \mathbb{F}^{m\times n}&\longrightarrow \mathbb{F}^{2m\times m} \\
(Y,Z) &\longrightarrow \begin{bmatrix}
S_0(Y,Z) \\ S_1(Y,Z)
\end{bmatrix} := \begin{bmatrix}
Y(bF+dE)^\star+EZ^\star \\
Y(aF+cE)^\star+FZ^\star
\end{bmatrix}.
\end{align*}
If the pair of matrices $(Y_0,Z_0)$ is a solution of the system of Sylvester equations $S_{(A,E,F)}(Y,Z)=\left[ \begin{smallmatrix} C_0 \\ C_1 \end{smallmatrix} \right]$, then $X_0=(Y_0+Z_0)/2$ is a solution of the system of $\star$-Sylvester equations  $T_{(A,E,F)}(X)= \left[ \begin{smallmatrix} C_0 \\ C_1 \end{smallmatrix} \right]$.
\end{theorem}
\begin{proof}
Assume that there exist matrices $Y_0,Z_0\in\mathbb{F}^{m\times n}$ satisfying the linear system of matrix equations $S_{(A,E,F)}(Y,Z)=\left[\begin{smallmatrix} C_0 \\ C_1 \end{smallmatrix}\right]$, i.e., 
\begin{align}
&Y_0(bF+dE)^\star+EZ_0^\star = C_0,\label{eq:Sys_eq1} \\
&Y_0(aF+cE)^\star+FZ_0^\star = C_1. \label{eq:Sys_eq2}
\end{align}
Applying the $(\cdot)^\star$ operator on both sides of \eqref{eq:Sys_eq1} and \eqref{eq:Sys_eq2}, and using that the pencil $\lambda C_1+C_0$ is $\mathbf{M}_A$-structured, we obtain that the pair of matrices $(Y_0,Z_0)$ also satisfies
\begin{align}
&Z_0E^\star + (bF+dE)Y_0^\star = bC_1+dC_0,\label{eq:Sys_eq3} \\ 
&Z_0F^\star + (aF+cE)Y_0^\star = aC_1+cC_0. \label{eq:Sys_eq4}
\end{align}

Let $X_0=(Y_0+Z_0)/2$.
To prove the desired result, we have to check that $X_0$ satisfies the equations $T_0(X_0)=C_0$ and $T_1(X_0)=C_1$.
For the first equation, using \eqref{eq:Sys_eq1}, we obtain
\[
T_0(X_0)=\frac{(Y_0+Z_0)}{2}(bF+dE)^\star + E\frac{(Y_0+Z_0)^\star}{2} = \frac{1}{2}\left( C_0+\overline{b}Z_0F^\star + \overline{d}Z_0E^\star+EY_0^\star \right).
\]
Then, from \eqref{eq:Sys_eq4} we get  $\overline{b}Z_0F^\star = a\overline{b}C_1+c\overline{b}C_0-a\overline{b}FY_0^\star - c\overline{b}EY_0^\star$, and, from \eqref{eq:Sys_eq3}, we get $\overline{d}Z_0E^\star = b\overline{d}C_1+d\overline{d}C_0-b\overline{d}FY_0^\star-d\overline{d}EY_0^\star$.
Substituting these expressions for $\overline{b}Z_0F^\star$ and $\overline{d}Z_0E^\star$ in the above equation, we obtain
\[
T_0(X_0)=\frac{1}{2}\left( C_0+\overline{b}Z_0F^\star + \overline{d}Z_0E^\star+EY_0^\star \right) = C_0,
\]
where we have used $a\overline{b}+b\overline{d}=0$ and $c\overline{b}+d\overline{d}=1$, which follows from $A\overline{A}=I_2$.
Therefore, the matrix $X_0$ satisfies the first matrix equation $T_0(X_0)=C_0$.
Proceeding in a similar way, it is not difficult to show that 
\[
T_1(X_0)=\frac{(Y_0+Z_0)}{2}(aF+cE)^\star + F\frac{(Y_0+Z_0)^\star}{2} = C_1.
\]
Thus, we conclude that $T_{(A,E,F)}(X_0)= \left[ \begin{smallmatrix} C_0 \\ C_1 \end{smallmatrix} \right]$, as we wanted to prove.
\end{proof}

To apply Theorem \ref{thm:from_T-Syl_to_Syl} for obtaining solutions of the systems of $\star$-Sylvester equations   \eqref{eq:fixed_point_it_X0} and \eqref{eq:fixed_point_it} solving, instead, a linear system of Sylvester equations, their right-hand-sides need to be the trailing  and leading coefficients of an $\mathbf{M}_A$-structured pencil.
It is clear that this is the case for the right-hand-side of  \eqref{eq:fixed_point_it_X0}, since $\lambda \Delta B_{22}+\Delta A_{22}$ is by assumption the $(2,2)$ block of an $\mathbf{M}_A$-structured pencil.
In Lemma \ref{lemma:right-hand-side}, we show that this is also true for the right-hand-side of \eqref{eq:fixed_point_it}.
\begin{lemma}\label{lemma:right-hand-side}
Let $A\in {\rm GL}(2,\mathbb{F})$ and $X\in\mathbb{F}^{kn\times (k+1)n}$.
If the pencils $\lambda M_1+M_0,\lambda \Delta B_{11}+\Delta A_{11}\in\mathbb{F}[\lambda]^{(k+1)n\times(k+1)n}$,  and $\lambda \Delta B_{22}+\Delta A_{22}\in\mathbb{F}[\lambda]^{kn\times kn}$ are $\mathbf{M}_A$-structured pencils, then the pencil $\lambda(\Delta B_{22}+X(M_1+\Delta B_{11})X^\star)+\Delta A_{22}+X(M_0+\Delta A_{11})X^\star$ is also $\mathbf{M}_A$-structured.
\end{lemma}
\begin{proof}
Let us introduce the notation $C_1:=\Delta B_{22}+X(M_1+\Delta B_{11})X^\star$ and $C_0:= \Delta A_{22}+X(M_0+\Delta A_{11})X^\star$.
The proof is immediate from the fact that the pencil
$\left[\begin{smallmatrix} \lambda(M_1+\Delta B_{11})+M_0+\Delta A_{11} & 0 \\ 0 & \lambda \Delta B_{22}+\Delta A_{22} \end{smallmatrix} \right]$ is $\mathbf{M}_A$-structured, and the fact that the pencil $\lambda C_1+C_0$ is the (2,2) block of
\[
\begin{bmatrix}
I_{(k+1)n} & 0 \\ X & I_{kn}
\end{bmatrix}
\begin{bmatrix}
 \lambda(M_1+\Delta B_{11})+M_0+\Delta A_{11} & 0 \\ 0 & \lambda \Delta B_{22}+\Delta A_{22}
\end{bmatrix}
\begin{bmatrix}
I_{(k+1)n} & X^\star \\ 0 & I_{kn}
\end{bmatrix},
\]
which is also $\mathbf{M}_A$-structured by Proposition \ref{prop:congruence}.
%To prove that the pencil $\lambda C_1+C_0$ is $\mathbf{M}_A$-structured we only have to check $C_1^\star=aC_1+cC_0$ and $C_0^\star = bC_1+dC_0$ (recall Example \ref{ex:M_A_pencil}).
%But this follows immediately from the $\mathbf{M}_A$-structure of the pencils $\lambda M_1+M_0$, $\lambda \Delta B_{11}+\Delta A_{11}$ and $\lambda \Delta B_{22}+\Delta A_{22}$, and we invite the reader to fill in the details.
\end{proof}

Theorem \ref{thm:from_T-Syl_to_Syl}, together with Lemma \ref{lemma:right-hand-side}, allows us to replace the fixed point iteration \eqref{eq:fixed_point_it_X0}-\eqref{eq:fixed_point_it} for getting a solution of \eqref{eq:gen_T-sylv} with the new iteration 
 \begin{align}
   \nonumber &\mbox{Solve for $(Y_0,Z_0)$ the system of Sylvester equations:} \\ \label{eq:fixed_point_it_Y0ZO}
&\left\{\begin{array}{l}
      Y_0 (b\widehat{F}_{kn}+d\widehat{E}_{kn} )^\star+\widehat{E}_{kn}Z_0^\star= -\Delta A_{22}\\
      Y_0 (a\widehat{F}_{kn}+c\widehat{E}_{kn})^\star+\widehat{F}_{kn}Z_0^\star= -\Delta B_{22} \end{array}\right. , \\
\nonumber &\mbox{and set }X_0 := (Y_0+Z_0)/2. \\
 \nonumber &\mbox{For $i\geq 1$, solve for $(Y_i,Z_i)$ the system of Sylvester equations:} \\ \label{eq:fixed_point_it_YZ}
&\left\{\begin{array}{l}
      Y_i (b\widehat{F}_{kn}+d\widehat{E}_{kn} )^\star+\widehat{E}_{kn}Z_i^\star= -\Delta A_{22}-f_{M_0+\Delta A_{11}}(X_{i-1})\\
      Y_i (a\widehat{F}_{kn}+c\widehat{E}_{kn})^\star+\widehat{F}_{kn}Z_i^\star= -\Delta B_{22}-f_{M_1+\Delta B_{11}}(X_{i-1}) \end{array}\right. , \\
 \nonumber &\mbox{and set }   X_i := (Y_i+Z_i)/2.
\end{align}

Observe that the linear systems of Sylvester equations \eqref{eq:fixed_point_it_Y0ZO}--\eqref{eq:fixed_point_it_YZ} are underdetermined since the number of entries of the pair $(Y_i,Z_i)$, i.e., the number of scalar unknowns, is $2k(k+1)n^2$ while the number of scalar equations is $2k^2n^2$.
Of course, this does not imply that the systems are consistent, so the next step is to obtain conditions on the norm of the perturbation pencil $\|\Delta \mathcal{L}(\lambda)\|_F$ that guarantee that the operator $S_{(A,\widehat{E}_{kn},\widehat{F}_{kn})}(Y,Z)$ introduced in Theorem \ref{thm:from_T-Syl_to_Syl} is surjective.
With this aim in mind, let us notice that a system of Sylvester matrix equations of the form
\begin{align}\label{eq:gen_sylv_linear}
\begin{split}
\left\{\begin{array}{l}
      Y (b(F_{kn}+\Delta B_{21})+d(-E_{kn}+\Delta A_{21}) )^\star+(-E_{kn} + \Delta A_{21})Z^\star= C_0\\
      Y (a(F_{kn} + \Delta B_{21})+c(-E_{kn}+\Delta A_{21}))^\star+(F_{kn} + \Delta B_{21})Z^\star= C_1,  \end{array}\right.
\end{split}
\end{align}
can be written, using the Kronecker product $\otimes$ and the $\vect(\cdot)$ operation, as the underdetermined standard linear system $(T_A+\Delta T_A)x = b$ given by
\begin{align}\label{eq:linear_operator_equation}
  &\left(
    \underbrace{\left[
      \begin{array}{c|c}
        (bF_{kn}-dE_{kn})\otimes I_{kn} & -I_{kn}\otimes E_{kn} \\ \hline
        (aF_{kn}-cE_{kn})\otimes I_{kn}& I_{kn}\otimes F_{kn}
      \end{array}
    \right]}_{=: T_A}
    + \right. \\ \nonumber
    &\left.\underbrace{
    \left[
      \begin{array}{c|c}
        (b\overline{\Delta B_{21}} +d\overline{\Delta A_{21}})\otimes I_{kn} & I_{kn}\otimes \Delta A_{21}\\\hline \\[-2.2 ex]
        (a\overline{\Delta B_{21}}+c\overline{\Delta A_{21}}) \otimes I_{kn}&I_{kn}\otimes \Delta B_{21}
      \end{array}
    \right]}_{=: \Delta T_A}
    \right)
    \underbrace{\left[
    \begin{array}{c}
      \mathrm{vec}(Y)\\\hline
      \mathrm{vec}(Z^\star)
    \end{array}
  \right]}_{=:x } = 
  \underbrace{\left[
    \begin{array}{c}
      \mathrm{vec}(C_0)\\\hline
      \mathrm{vec}(C_1)
    \end{array}
  \right]}_{=:b}
  \>,
\end{align}
where we have also used that all the entries of the matrices $E_{kn}$ and $F_{kn}$ are real.
Then, the bilinear operator $S_{(A,\widehat{E}_{kn},\widehat{F}_{kn})}(Y,Z)$ is surjective if and only if the matrix $T_A+\Delta T_A$ has full row rank.

In Lemma \ref{lemma:min_singular_val} we show that the matrix $T_A$ has full row rank. 
Additionally, we also provide a formula for its minimal singular value.
The proof of Lemma \ref{lemma:min_singular_val} is rather long, so it is postponed to Appendix \ref{appendix:proof}.
\begin{lemma}\label{lemma:min_singular_val}
Let $A\in {\rm GL}(2,\mathbb{F})$ be any of the matrices in Table \ref{table:1}.
Then, the matrix $T_A$ in \eqref{eq:linear_operator_equation} has full row rank, and its minimal singular value is given by
$\sigma_{\mathrm{min}}(T_A)=2\sin(\pi/(4k))$.
\end{lemma}

Lemma \ref{lemma:min_singular_val} implies that if $\|\Delta T_A\|_2$ is small enough, then $T_A+ \Delta T_A$ has also full row rank and the linear system \eqref{eq:linear_operator_equation} is consistent for any right-hand-side, or, equivalently, the bilinear operator $S_{(A,\widehat{E}_{kn},\widehat{F}_{kn})}(Y,Z)$  is surjective.
In Lemma \ref{lemma:linear_bound}, we bound the norm of the minimum 2-norm solution of \eqref{eq:linear_operator_equation} or, equivalently, of the minimum Frobenius norm solution of the equation \eqref{eq:gen_sylv_linear}, since $\|[\mathrm{vec}(Y)^T, \mathrm{vec}(Z^\star)^T ]^T \|_2  = \|(Y,Z) \|_F$ (recall the definition of the Frobenius norm of a pair of matrices in \eqref{eq:norm_pair}).
We omit the proof of Lemma \ref{lemma:linear_bound}, since it is identical to the proof of \cite[Lemma 6.6]{minimal_pencils}.
\begin{lemma}\label{lemma:linear_bound}
Let $(T_A+\Delta T_A)x= b$ be the underdetermined linear system  \eqref{eq:linear_operator_equation}, and let us assume that $\sigma_{\min}(T_A)>\|\Delta T_A\|_2$.
Then $(T_A+\Delta T_A)x= b$ is consistent and its minimum norm solution $(Y_0, Z_0)$ satisfies
  \begin{equation}\label{eq:bound_YZ}
    \|(Y_0,Z_0)\|_F\leq \frac{1}{\delta} \|(C_0,C_1)\|_F,
  \end{equation}
 where $ \delta:=\sigma_{\rm min}(T_A)-\|\Delta T_A\|_2$.
\end{lemma}
%\begin{proof}
%From Weyl's perturbation theorem for singular values \cite[Theorem 3.3.16]{Horn}, we get $\sigma_{\rm min}(T_A+\Delta T_A)\geq \sigma_{\rm min}(T_A)-\|\Delta T_A\|_2>0$.
%Therefore, the matrix $T_A+\Delta T_A$ has full row rank and the linear system \eqref{eq:linear_operator_equation} is consistent. Its minimum norm solution, $(Y_0, Z_0)$, is given by $(T_A+\Delta T_A)^\dagger b$, where $(T_A+\Delta T_A)^\dagger$ denotes the Moore-Penrose pseudoinverse of $T_A+\Delta T_A$.
%Then, notice
%  \begin{align*}
%    \|(Y_0,Z_0)\|_F \leq &\|(T_A+\Delta T_A)^\dagger \|_2\|(C_0,C_1)\|_F = \frac{1}{\sigma_{\rm min}(T_A+\Delta T_A)}\|(C_0,C_1)\|_F \\
%   \leq &  \frac{1}{\sigma_{\rm min}(T_A)-\|\Delta T_A\|_2}\|( C_0,C_1)\|_F,
%  \end{align*}
%  and the result is established.
%\end{proof}

Since the quantity $\sigma_{\mathrm{min}}(T_A)-\|\Delta T_A\|_2$ plays an important role in the rest of our analysis, we obtain in Lemma \ref{lemm:deltabound} a tractable lower bound on it.
The proof of this lemma is almost identical to the one for \cite[Lemma 6.7]{minimal_pencils}, so it is omitted.
\begin{lemma} \label{lemm:deltabound} Let $T_A$ and $\Delta T_A$ be the matrices in \eqref{eq:linear_operator_equation} with $A\in {\rm GL}(2,\mathbb{F})$  any of the matrices in Table \ref{table:1}, let $\Delta \mathcal{L}(\lambda)$ be the pencil in \eqref{eq:blocksofdeltaL}.
 If $\|\Delta \mathcal{L}(\lambda)\|_F < 1/(3k)$, then
\[
\sigma_{\rm min}(T_A)-\|\Delta T_A\|_2 \geq \frac{\pi}{4 k} \, (1 - 3 k \|\Delta\mathcal{L}(\lambda)\|_F ) >0 \, .
\]
\end{lemma}

%\begin{remark}
%Note that the bound in Lemma \ref{lemm:deltabound} is pessimistic and that ``$g$'' can be replaced in this paper by ``$k$'' just modifying a bit the argument in \cite{minimal_pencils}.
%However, the current version of Lemma \ref{lemm:deltabound} allows us to reuse significant parts of the analysis developed in \cite{minimal_pencils}, making the analysis in this paper shorter.
%Since the replacement of ``$g$'' by ``$k$'' only improves the bounds by factors of order 2, we have preferred to simplify the analysis.
%\end{remark}

Finally, we show in Theorem \ref{thm:gen_sylvester_solution} and its corollary Theorem \ref{thm:finalofstep1} that the fixed point iteration \eqref{eq:fixed_point_it_X0}--\eqref{eq:fixed_point_it} or, equivalently, the iteration \eqref{eq:fixed_point_it_Y0ZO}--\eqref{eq:fixed_point_it_YZ} by choosing minimum norm solutions $(Y_i,Z_i)$ at each step converges to a solution $X$ of the system of quadratic $\star$-Sylvester-like matrix equations \eqref{eq:gen_T-sylv} such that $\|X\|_F \lesssim \|\Delta \mathcal{L}(\lambda)\|_F$, whenever $\|\Delta \mathcal{L}(\lambda)\|_F$ is properly upper bounded.
The proof of Theorem \ref{thm:gen_sylvester_solution} follows  closely those by Stewart  \cite[Theorem 5.1]{Stewart} and Dopico, Lawrence, P\'erez and Van Dooren  \cite[Theorem 6.8]{minimal_pencils}.
\begin{theorem}\label{thm:gen_sylvester_solution}
  There exists a solution $X$ of the quadratic system of $\star$-Sylvester-like matrix equations \eqref{eq:gen_T-sylv} satisfying
\begin{equation}\label{eq:norm_solution}
  \|X\|_F\leq 2\frac{\theta}{\delta}\>,
\end{equation}
whenever
\begin{equation}\label{eq:condition_convergence}
 \delta>0 \quad \mbox{and} \quad \frac{\theta\omega}{\delta^2}<\frac{1}{4}\>,
\end{equation}
where $\delta = \sigma_{\rm min}(T_A)-\|\Delta T_A\|_2$, $\theta :=\|(\Delta A_{22},\Delta B_{22})\|_F$, and $\omega:=\|(M_0+\Delta A_{11},M_1+\Delta B_{11})\|_F$.	
\end{theorem}
\begin{proof}
Lemma \ref{lemma:linear_bound} and the hypothesis $\delta >0$ guarantee that the linear system of matrix equations \eqref{eq:gen_sylv_linear} is consistent for any right-hand side.
Let the minimum norm solution of \eqref{eq:fixed_point_it_Y0ZO} be denoted by $(Y_0,Z_0)$, and set $X_0:=(Y_0+Z_0)/2$. 
From Theorem \ref{thm:from_T-Syl_to_Syl} and Lemma \ref{lemma:linear_bound} we get that the matrix $X_0$ is a solution of the matrix equation \eqref{eq:fixed_point_it_X0} such that
\[
    \|X_0\|_F\leq \|(Y_0,Z_0)\|_F \leq \frac{1}{\delta}\|(\Delta A_{22}, \Delta B_{22})\|_F = \frac{\theta}{\delta}=:\rho_0.
\]
Then, let us define a sequence $\{X_i\}_{i=0}^\infty$ of matrices as follows: for $i > 0$, the matrix $X_i$ is defined as $X_i:=(Y_i+Z_i)/2$, where the pair of matrices  $(Y_i, Z_i)$ denotes the minimum Frobenius norm solution of the underdetermined system \eqref{eq:fixed_point_it_YZ}.
Clearly, we have $\|X_i\|_F\leq \|(Y_i,Z_i)\|_F$.
Moreover, vectorizing \eqref{eq:fixed_point_it_YZ} and using the matrix $T_A + \Delta T_A$ defined in \eqref{eq:linear_operator_equation}, we get
\begin{equation} \label{eq:fixed_it_vec}
    \left[
      \begin{array}{c}
        \mathrm{vec}(Y_{i})\\
        \mathrm{vec}(Z_{i}^\star)
      \end{array}
    \right]
    =
    \left[
      \begin{array}{c}
        \mathrm{vec} (Y_0)\\
        \mathrm{vec} (Z_0^\star)
      \end{array}
    \right]
    -(T_A+\Delta T_A)^\dagger
      \left[
        \begin{array}{c}
          \mathrm{vec}(X_{i-1}(M_0+\Delta A_{11})X_{i-1}^\star)\\
          \mathrm{vec}(X_{i-1}(M_1+\Delta B_{11})X_{i-1}^\star)
        \end{array}
      \right]\>.
\end{equation}
We claim that the sequence $\{X_i\}_{i=0}^\infty$ converges to a solution $X$ of \eqref{eq:gen_T-sylv}  satisfying \eqref{eq:norm_solution}.
To prove this, we first show that the sequence  $\{\|X_i\|_F\}_{i=0}^\infty$ is a bounded sequence.
If $\|X_{i-1}\|_F\leq \|(Y_{i-1},Z_{i-1})\|_F \leq \rho_{i-1}$, then we have from \eqref{eq:fixed_it_vec} that
  \begin{align*}
    \|X_i\|_F\leq &\|(Y_{i},Z_{i})\|_F  \leq  \|(Y_0,Z_0)\|_F \\
     &+\|(T_A+\Delta T_A)^\dagger \|_2\|\|(Y_{i-1},Z_{i-1})\|_F^{2}\|(M_0+\Delta A_{11},M_1+\Delta B_{11})\|_F\\
    \leq &\rho_0+\rho_{i-1}^2\omega\delta^{-1}=:\rho_{i}\>.
  \end{align*}
The quantity  $\rho_{i}$ in the equation above may be written as $\rho_{i}=\rho_0(1+\kappa_i)$, where $\kappa_i$ satisfies the recursion
\begin{equation} \label{eq:fixedstewart}
  \left\{ \begin{array}{l}
    \kappa_1=\rho_0\omega\delta^{-1}=\theta \omega \delta^{-2},\\
    \kappa_{i+1}=\kappa_1(1+\kappa_i)^2\>.
    \end{array}\right.
\end{equation}
As it is shown in the proofs of \cite[Theorem 6.8]{minimal_pencils} and \cite[Theorem 5.1]{Stewart}, if $\kappa_1<1/4$, then $\lim_{i \rightarrow \infty} \kappa_i = \kappa$, where $\kappa$ is given by
\[
\kappa =\lim_{i\rightarrow \infty} \kappa_i = \frac{2\kappa_1}{1-2\kappa_1+\sqrt{1-4\kappa_1}} < 1,
  \]
and $\kappa_i <\kappa$ for all $i\geq 1$.
Thus, the norms of the elements of the sequence $\{X_i\}_{i=0}^\infty$ are bounded as
  \begin{equation} \label{eq:bound_X}
    \|X_i\|_F\leq \rho :=\lim_{i\rightarrow \infty}{\rho_i}=\rho_0(1+\kappa)\>.
  \end{equation}
  
We now show that the sequence $\{ X_i\}_{i=0}^\infty$ converges provided that $2\delta^{-1}\omega \rho<1$, which is guaranteed by \eqref{eq:condition_convergence}.
To this purpose, let $S_i:=(Y_i-Y_{i-1},Z_i-Z_{i-1})$. 
Then, notice that \eqref{eq:fixed_it_vec} implies
\begin{alignat*}{3}
&\|S_i\|_F \\ & \leq \|(T_A+\Delta T_A)^\dagger\|_2 \,
\left\| \begin{bmatrix}
{\rm vec}\left(X_{i-1}(M_0+\Delta A_{11})X_{i-1}^\star-X_{i-2}(M_0+\Delta A_{11})X_{i-2}^\star \right) \\
{\rm vec}\left(X_{i-1}(M_1+\Delta B_{11})X_{i-1}^\star-X_{i-2}(M_1+\Delta B_{11})X_{i-2}^\star \right)
\end{bmatrix} \right\|_2 \\
&\leq \delta^{-1} \,
\left\| \left[ \begin{matrix}
{\rm vec}\left((X_{i-1}-X_{i-2})(M_0+\Delta A_{11})X_{i-1}^\star \right) \\
{\rm vec}\left((X_{i-1}-X_{i-2})(M_1+\Delta B_{11})X_{i-1}^\star \right) 
\end{matrix} \right. \right. \\
&\hspace{5cm}\left.\left. \begin{matrix}
 + {\rm vec}\left(X_{i-2}(M_0+\Delta A_{11})(X_{i-1}-X_{i-2})^\star\right)\\
 + {\rm vec}\left(X_{i-2}(M_1+\Delta B_{11})(X_{i-1}-X_{i-2})^\star\right) 
\end{matrix}\right]\right\|_2
\\
&\leq 2\delta^{-1}\omega\rho \|X_{i-1}-X_{i-2}\|_F\leq 2\delta^{-1}\omega \rho \|S_{i-1}\|_F.
\end{alignat*}
%%%%%%%%%%%%%%%%%%%%%%%%%%%
%\left\| \begin{bmatrix}
%{\rm vec}\left((X_{i-1}-X_{i-2})(M_0+\Delta A_{11})X_{i-1}^\star + X_{i-2}(M_0+\Delta A_{11})(X_{i-1}-X_{i-2})^\star\right) \\
%{\rm vec}\left((X_{i-1}-X_{i-2})(M_1+\Delta B_{11})X_{i-1}^\star + X_{i-2}(M_1+\Delta B_{11})(X_{i-1}-X_{i-2})^\star\right) 
%\end{bmatrix
%%%%%%%%%%%%%%%%%%%%%%%%%%%
%\begin{align*}
%\|S_i\|_F & \leq \|(T_A+\Delta T_A)^\dagger\|_2 \,
%\left\| \begin{bmatrix}
%{\rm vec}\,\left( Y_i(M_0+\Delta A_{11})Z_i - Y_{i-1}(M_0+\Delta A_{11})Z_{i-1} \right) \\
%{\rm vec}\,\left( Y_i(M_1+\Delta B_{11})Z_i - Y_{i-1}(M_1+\Delta B_{11})Z_{i-1} \right)
%\end{bmatrix} \right\|_2 \\
%&\leq \delta^{-1} \,
%\left\| \begin{bmatrix}
%{\rm vec}\,\left(S_{i-1}^{(Y)}(M_0+\Delta A_{11})Z_i + Y_{i-1}(M_0+\Delta A_{11})S_{i-1}^{(Z)} \right) \\
%{\rm vec}\,\left( S_{i-1}^{(Y)}(M_1+\Delta B_{11})Z_i + Y_{i-1}(M_1+\Delta B_{11})S_{i-1}^{(Z)} \right)
%\end{bmatrix} \right\|_2 \\
%&\leq 2\delta^{-1}\omega \rho \|S_{i-1}\|_F.
%\end{align*}
Since $2\delta^{-1}\omega \rho<1$, we get that $\{(Y_i,Z_i)\}_{i=0}^\infty$ is a Cauchy sequence and, therefore, it has a limit $(Y,Z):=\lim_{i\rightarrow \infty} (Y_i,Z_i)$.
Thus, the matrix $X:=(Y+Z)/2=\lim_{i\rightarrow \infty}(Y_i+Z_i)/2=\lim_{i\rightarrow \infty} X_i$ exists.
Finally, we show that the matrix $X$ is a solution of \eqref{eq:gen_T-sylv}  satisfying \eqref{eq:norm_solution}.
First, since the sequence $\{(Y_i,Z_i)\}_{i=0}^\infty$ satisfies \eqref{eq:fixed_it_vec} and, so, \eqref{eq:fixed_point_it_YZ}, we have that the sequence $\{X_i=(Y_i+Z_i)/2\}_{i=0}^\infty$ satisfies \eqref{eq:fixed_point_it} as a consequence of Theorem \ref{thm:from_T-Syl_to_Syl} and  Lemma \ref{lemma:right-hand-side}.
Then, by taking limits in both sides of \eqref{eq:fixed_point_it}, we get that $X$ is a solution of \eqref{eq:gen_T-sylv}.
We conclude the proof just noticing that \eqref{eq:bound_X} implies $\|X\|_F\leq \rho_0(1+\kappa)<2\rho_0=2\delta^{-1}\theta$.
\end{proof}

We complete the first step of the structured backward error analysis with Theorem \ref{thm:finalofstep1}.
Its proof follows from Lemma \ref{lemm:deltabound}, Theorem \ref{thm:gen_sylvester_solution} and norm inequalites, and it is identical to its unstructured counterpart \cite[Theorem 6.9]{minimal_pencils}, so it is omitted.
\begin{theorem} \label{thm:finalofstep1}
Let $\mathcal{L}(\lambda)$ be an $\mathbf{M}_A$-structured block Kronecker pencil as in \eqref{eq:structured-Kron-pencil}, where $A$ is any of the matrices in Table \ref{table:1}, and let $\Delta \mathcal{L}(\lambda)$ be any pencil with the same size and structure as $\mathcal{L}(\lambda)$  such that
\begin{equation} \label{eq:boundL1}
\|\Delta \mathcal{L}(\lambda)\|_F < \left(\frac{\pi}{16} \right)^2 \, \frac{1}{k^2} \, \frac{1}{1 + \|\lambda M_1 + M_0\|_F}.
\end{equation}
Then, there exists a matrix $X\in\FF^{kn\times (k+1)n}$ that satisfies
\begin{equation} \label{eq:easyboundCD}
\|X\|_F \leq \frac{3k}{1-3k \|\Delta\mathcal{L}(\lambda)\|_F} \, \|\Delta \mathcal{L}(\lambda)\|_F,
\end{equation}
and the equality \eqref{eq:back_to_zero} with
\begin{equation}\label{eq:boundL12tilde}
\|\Delta \widetilde{\mathcal{L}}_{21}(\lambda)\|_{F} \leq  
 \|\Delta \mathcal{L}(\lambda)\|_F 
\left(1 + \frac{3k}{1-3k \|\Delta \mathcal{L}(\lambda)\|_F} \, (\|\lambda M_1 + M_0\|_F + \|\Delta \mathcal{L}(\lambda)\|_F) \right).
\end{equation}
\end{theorem}

\subsection{Second step: proving that $\mathcal{L}(\lambda)+\Delta \mathcal{\widetilde{L}}(\lambda)$ in \eqref{eq:back_to_zero} is an $\mathbf{M}_A$-structured strong block minimal bases pencil}\label{sec:second_step}

The aim of this section is to obtain bounds on $\|\Delta \mathcal{\widetilde{L}}(\lambda)\|_F$ that ensure  the pencil \eqref{eq:back_to_zero} is an $\mathbf{M}_A$-structured strong block  minimal bases pencil. 
To prove this, we rely heavily on some important minimal bases perturbations results in \cite[Section 6.2]{minimal_pencils}.
In particular, we will use \cite[Theorem 6.18]{minimal_pencils}, which is stated below for completeness.

\begin{theorem} \label{thm:finalofstep2} Let $L_k (\lambda)$ and $\Lambda_k (\lambda)^T$ be the pencil and the row vector polynomial defined in \eqref{eq:Lk} and \eqref{eq:Lambda}, respectively, and let $\Delta \widetilde{\mathcal{L}}_{21} (\lambda)$ be any pencil of size $k n \times (k + 1) n$ such that
\begin{equation} \label{eq:final21pertbound}
\|\Delta \widetilde{\mathcal{L}}_{21} (\lambda) \|_F < \frac{\pi}{12 (k +1)^{3/2}}.
\end{equation}
Then, there exists a matrix polynomial $\Delta R_k (\lambda)^T$ with size $n \times (k + 1) n$ and grade $k$ such that
\begin{enumerate}
\item[\rm (a)]  $L_{k}(\lambda)\otimes I_n+\Delta \widetilde{\mathcal{L}}_{21}(\lambda)$ and $\Lambda_{k}(\lambda)^T\otimes I_n+\Delta R_k (\lambda)^T$ are dual minimal bases, with all the row degrees of the former equal to $1$ and with all the row degrees of the latter equal to $k$, and
\smallskip
\item[\rm (b)] $\displaystyle \|\Delta R_k (\lambda) \|_F \leq
\frac{6 \, \sqrt{2} \, (k+1)}{\pi} \, \|\Delta \widetilde{\mathcal{L}}_{21} (\lambda) \|_F \,
<  \frac{1}{\sqrt{2}}$.
\end{enumerate}
\end{theorem}

By using Theorem \ref{thm:finalofstep2}, together with Theorem \ref{thm:minimal_basis_Mobius}, and the definition of $\mathbf{M}_A$-structured strong block minimal bases pencils in Definition \ref{def:structured-minimal-bases-pencil}, we prove the final result of this section.
 \begin{theorem} \label{thm:corfinalofstep2} 
Let $A\in {\rm GL}(2,\mathbb{F})$ be any of the matrices listed in Table \ref{table:1}, let $\mathcal{L}(\lambda)+\Delta \widetilde{\mathcal{L}}(\lambda)$ be the pencil in \eqref{eq:back_to_zero}.
If
\[
\|\Delta \widetilde{\mathcal{L}}_{21} (\lambda) \|_F < \frac{\pi}{12 \, (k+1)^{3/2}},
\]
then $\mathcal{L}(\lambda)+\Delta \widetilde{\mathcal{L}}(\lambda)$ is an $\mathbf{M}_A$-structured strong block minimal bases pencil. 
Moreover, there exists a matrix polynomial $\Delta R_k (\lambda)^T$ of grade $k$ such that $\Lambda_{k}(\lambda)^T\otimes I_n+\Delta R_k (\lambda)^T$ is a minimal basis dual to $ L_k(\lambda)\otimes I_n+\Delta \mathcal{\widetilde{L}}_{21}(\lambda)$ with all its row degrees equal to $k$, and
\begin{equation}\label{eq:bound_dual}
 \|\Delta R_k (\lambda) \|_F = \|\mathbf{M}_A[\Delta R_k] (\lambda) \|_F  \leq
\frac{6 \, \sqrt{2} \, (k+1)}{\pi} \, \|\Delta \widetilde{\mathcal{L}}_{21} (\lambda) \|_F \,
<  \frac{1}{\sqrt{2}} \, .
\end{equation}
\end{theorem}
\begin{proof}
By  Theorem \ref{thm:finalofstep2}, the matrix pencil $L_k(\lambda)\otimes I_n+\Delta \widetilde{\mathcal{L}}_{21}(\lambda)$ is a minimal basis with all its row degrees equal to 1, and there exists a matrix polynomial $\Delta R_k(\lambda)$ such that $\Lambda_k(\lambda)^T\otimes I_n +\Delta R_k(\lambda)^T$ is a dual minimal basis of $L_k(\lambda)\otimes I_n+\Delta \widetilde{\mathcal{L}}_{21}(\lambda)$ with all its row degrees equal to $k$.
Therefore, by Definition \ref{def:structured-minimal-bases-pencil}, the pencil $\mathcal{L}(\lambda)+\Delta \widetilde{\mathcal{L}}(\lambda)$ in \eqref{eq:back_to_zero} is an $\mathbf{M}_A$-structured strong block minimal bases pencil.
To finish the proof, we only need to prove the upper bound \eqref{eq:bound_dual}.
Indeed, the upper bound for $\|\Delta R_k(\lambda)\|_F$ follows from part-(b) in Theorem \ref{thm:finalofstep2}.
Moreover, since $A$ is any of the matrices in Table \ref{table:1}, it is easily checked that $\| \mathbf{M}_A[\Delta R_k](\lambda) \|_F = \|\Delta R_k(\lambda)\|_F$, and the result is established.
\end{proof} 
 
\subsection{Third step: Mapping structured perturbations to a structured block Kronecker pencil onto the structured matrix polynomial}
Combining the results in Sections \ref{sec:first_step} and \ref{sec:second_step}, in this section we finish the structured backward error analysis of structured odd-degree polynomial eigenvalue problems solved using structured block Kronecker pencils.
The main result is Theorem \ref{thm:back_error_main_theorem}, whose consequences are, then, discussed in Corollary \ref{cor:FINperturbation} and Remark \ref{remark:FIN}.

In order to simplify the proof of Theorem \ref{thm:back_error_main_theorem}, we present first the following lemma.
\begin{lemma}\label{lemma:FIN}
Let $P(\lambda)$ and $P(\lambda) + \Delta P(\lambda)$ be the matrix polynomials in \eqref{eq:poly_section6} and \eqref{eq:poly_perturbed}, respectively, where $A$ is one of the matrices in Table \ref{table:1}, and write $M(\lambda)=\lambda M_1+M_0$. 
If the matrix polynomial $\Delta R_k (\lambda)$ satisfies $\|\Delta R_k (\lambda)\|_F < 1/\sqrt{2}$, then
\[
\|\Delta P(\lambda)\|_F \leq \sqrt{k+1} \left(5 \|\Delta \mathcal{L}_{11}(\lambda) \|_F + 4  \|\lambda M_1 + M_0\|_F  \|\Delta R_k (\lambda)\|_F\right) \, ,
\]
where $g = 2k + 1$.
\end{lemma}
\begin{proof}
First, following the notation introduced in the proof of \cite[Lemma 6.20]{minimal_pencils}, for brevity, we use the notation $\Lambda_{k n}  := \Lambda_{k}(\lambda) \otimes I_n$ and $\Lambda_{k n}^\star := \Lambda_{k}(\lambda)^\star\otimes I_n$, and omit the dependence on $\lambda$ of some matrix polynomials.
Then, note that, since $A$ is one of the matrices in Table \ref{table:1}, we have $\|\Lambda_{k n}\|_F=\| \mathbf{M}_{A}[\Lambda_{k n}] \|_F$ and $\|\Delta R_k \|_F=\| \mathbf{M}_{A}[\Delta R_k] \|_F$. 
 From \eqref{eq:poly_section6} and \eqref{eq:poly_perturbed}, we get that
\begin{align} \nonumber
\Delta P (\lambda) = & \mathbf{M}_{A}[\Delta R_k]^\star(\lambda M_1 + M_0) \Lambda_{k n} + \mathbf{M}_{A}[\Lambda_{k n}]^\star \Delta \mathcal{L}_{11} \Lambda_{k n} + \mathbf{M}_{A}[\Delta R_k]^\star  \Delta\mathcal{L}_{11}  \Lambda_{k n} \\ \nonumber & + \mathbf{M}_{A}[\Lambda_{kn}]^\star (\lambda M_1 + M_0) \Delta R_k +
\mathbf{M}_{A}[\Delta R_{k}]^\star (\lambda M_1 + M_0) \Delta R_k \\ & + \mathbf{M}_{A}[\Lambda_{kn}]^\star \Delta\mathcal{L}_{11}  \Delta R_k + \mathbf{M}_{A}[\Delta R_k]^\star \Delta\mathcal{L}_{11}  \Delta R_k \, . \label{eq:longDeltaP}
\end{align}
The result follows from bounding the Frobenius norm of each of the terms in the right-hand side of \eqref{eq:longDeltaP}, using  Lemma \ref{lemma:normsproducts}, together with  $\|\Delta R_k \|_F  < 1/\sqrt{2}$ and $\| \mathbf{M}_{A}[\Delta R_k] \|_F < 1/\sqrt{2}$ in those terms that are not linear in $\Delta \mathcal{L}_{11} $, $\Delta R_k $, and $\mathbf{M}_{A}[\Delta R_k]$ for bounding them with linear terms.
In particular, note that Lemma \ref{lemma:normsproducts} implies that for any matrix polynomial $Z(\lambda)$ of grade $t$ and any $A$ in Table \ref{table:1} we have
\begin{align*}
\|\mathbf{M}_A[\Lambda_{k n}]^\star Z(\lambda)\|_F =
\|\mathbf{M}_A[\Lambda_{k n}^T\mathbf{M}_{A^{-1}}[Z]]\|_F\leq & \min\{\sqrt{k+1},\sqrt{t+1}\}\|\mathbf{M}_{A^{-1}}[Z]\|_F \\
 \leq &\min \{\sqrt{k+1},\sqrt{t+1}\}\|Z\|_F.
\end{align*}
With this observation, bounding all the terms of the right-hand-side of \eqref{eq:longDeltaP} is elementary but rather long, so we invite the reader to complete the proof.
\end{proof}

Finally, we are at the position of stating and proving the main result of this section, namely, the perturbation of structured block Kronecker pencils result in Theorem \ref{thm:back_error_main_theorem}.
\begin{theorem}\label{thm:back_error_main_theorem}
Let $P(\lambda) = \sum_{i=0}^g P_i \lambda^i \in \mathbb{F}[\lambda]^{n\times n}$ be a structured ((skew)-symmetric, (anti)-palindromic, or alternating) matrix polynomial and let $\mathcal{L}(\lambda)$ be a structured block Kronecker pencil as in \eqref{eq:structured-Kron-pencil}, where the matrix $A$ is one of the matrices in Table \ref{table:1} and it is chosen to guarantee that $\mathscr{S}(P)=\mathscr{S}(\mathcal{L})$, with $g = 2k + 1$ and such that
$P(\lambda) = (\mathbf{M}_{A}[\Lambda_k](\lambda)^\star\otimes I_n)(\lambda M_1+M_0)(\Lambda_k(\lambda)\otimes I_n)$, where $\lambda M_1 + M_0$ is the $(1,1)$-block in the natural partition of $\mathcal{L}(\lambda)$ and $\Lambda_k (\lambda)$ is the vector polynomial in \eqref{eq:Lambda}.
 If $\Delta \mathcal{L}(\lambda)$ is any pencil with the same size and structure as $\mathcal{L}(\lambda)$ and such that
\begin{equation} \label{eq:Lfinalbound}
\|\Delta \mathcal{L}(\lambda)\|_F < \left(\frac{\pi}{16} \right)^2 \, \frac{1}{(k+1)^{5/2}} \, \frac{1}{1 + \|\lambda M_1 + M_0\|_F},
\end{equation}
then $\mathcal{L}(\lambda) + \Delta \mathcal{L}(\lambda)$ is a strong linearization of a matrix polynomial $P(\lambda) + \Delta P(\lambda)$ with grade $g$ and such that
\[
\frac{\|\Delta P(\lambda)\|_F}{\|P(\lambda)\|_F} \leq 68\,  (k+1)^{5/2} \frac{\|\mathcal{L}(\lambda) \|_F}{\|P(\lambda)\|_F} \, (1+ \|\lambda M_1 + M_0\|_F + \|\lambda M_1 + M_0\|_F^2) \,
\frac{\|\Delta \mathcal{L}(\lambda) \|_F}{\|\mathcal{L}(\lambda) \|_F} \, ,
\]
and $\mathscr{S}(\Delta P)=\mathscr{S}(P)$.
In addition, the right minimal indices of $\mathcal{L}(\lambda) + \Delta \mathcal{L}(\lambda)$ are those of $P(\lambda) + \Delta P(\lambda)$ shifted by $k$, and the left minimal indices of $\mathcal{L}(\lambda) + \Delta \mathcal{L}(\lambda)$ are those of $P(\lambda) + \Delta P(\lambda)$ shifted also by $k$.
\end{theorem}
\begin{proof}
Notice that the condition \eqref{eq:Lfinalbound} implies \eqref{eq:boundL1}, so we can apply Theorem \ref{thm:finalofstep1} to prove that the pencil $\mathcal{L}(\lambda)+\Delta \mathcal{L}(\lambda)$ is $\star$-congruent to the pencil $\mathcal{L}(\lambda)+\Delta \widetilde{\mathcal{L}}(\lambda)$ in \eqref{eq:back_to_zero}.
Since $\star$-congruences are strict equivalences, both pencils have the same complete eigenstructures.
By using \eqref{eq:boundL12tilde} together with $3k \|\Delta \mathcal{L}(\lambda)\|_F < 1/2$, which is implied by  \eqref{eq:Lfinalbound}, we get
the following upper bound
\begin{align} \label{eq:aux1boundfinal}
\|\Delta \widetilde{\mathcal{L}}_{21}(\lambda)\|_{F} 
&\leq \|\Delta \mathcal{L}(\lambda)\|_F \left(2 + 6 \, k \, \|\lambda M_1 + M_0\|_F  \right) \, \\
& \leq 6 \left(\frac{\pi}{16} \right)^2 \, \frac{1}{(k+1)^{3/2}} < \frac{\pi}{12}  \, \frac{1}{(k+1)^{3/2}} \,. \nonumber
\end{align}
The above upper bound allows us to apply Theorem \ref{thm:corfinalofstep2} to the pencil $\mathcal{L}(\lambda)+\Delta \widetilde{\mathcal{L}}(\lambda)$.
Thus, the pencil $\mathcal{L}(\lambda)+\Delta \widetilde{\mathcal{L}}(\lambda)$ is an $\mathbf{M}_A$-structured strong block minimal bases pencil, which, by Theorem \ref{thm:structured-minimal-basis-pencil}, is a strong linearization of the matrix polynomial $P(\lambda)+\Delta P(\lambda)$ in \eqref{eq:poly_perturbed} with $\mathscr{S}(\Delta P)=\mathscr{S}(P)$.
Furthermore, Theorem \ref{thm:structured-minimal-basis-pencil} also implies that the right and left minimal indices of  $\mathcal{L}(\lambda)+\Delta \widetilde{\mathcal{L}}(\lambda)$ and, since they are strictly equivalent, the ones of $\mathcal{L}(\lambda)+\Delta \mathcal{L}(\lambda)$, are those of $P(\lambda)+\Delta P(\lambda)$ shifted by $k$. 
It only remains to obtain the upper bound for $\|\Delta P(\lambda)\|_F/\|P(\lambda)\|_F$. 
But this follows from combining Theorem \ref{thm:corfinalofstep2} with \eqref{eq:aux1boundfinal} to obtain 
\[
\|\Delta R_k (\lambda) \|_F = \|\mathbf{M}_{A}[\Delta R_k](\lambda) \|_F  \leq 17 \, (k+1)^2 \,
\|\Delta \mathcal{L}(\lambda)\|_F \left(1 +  \|\lambda M_1 + M_0\|_F  \right),
\]
and, then, combining the above upper bound with Lemma \ref{lemma:FIN}.
\end{proof}

Recall that our main goal is to study whether solving a  SPEP or a SCPE  applying a structurally backward stable algorithm (like the palindromic-QR or the structured staircase algorithm) to a structured  block Kronecker pencil is structurally global backward stable from the point of view of the polynomial or not.
In view of Theorem \ref{thm:back_error_main_theorem}, the structured backward stability is guaranteed when the constant
\begin{equation}  \label{eq:C_pl}
C_{P,\mathcal{L}} := 68\,  (k+1)^{5/2} \frac{\|\mathcal{L}(\lambda) \|_F}{\|P(\lambda)\|_F} \, (1+ \|\lambda M_1 + M_0\|_F + \|\lambda M_1 + M_0\|_F^2)
\end{equation}
is a moderate number.
To help us to study the size of \eqref{eq:C_pl}, we present Lemma \ref{lemma:quotient}.
Although Lemma \ref{lemma:quotient} is similar to \cite[Lemma 6.24]{minimal_pencils}, we reprove it in a simpler way that is more adequate in our structured setting.
\begin{lemma} \label{lemma:quotient} Let $P(\lambda) = \sum_{i=0}^g P_i \lambda^i \in \mathbb{F}[\lambda]^{n\times n}$ be a structured matrix polynomial and let $\mathcal{L}(\lambda)$ be a structured block Kronecker pencil as in \eqref{eq:structured-Kron-pencil}, where the matrix $A$ is one of the matrices in Table \ref{table:1} and it is chosen to guarantee that $\mathscr{S}(P)=\mathscr{S}(\mathcal{L})$, with $g = 2k + 1$ and such that
$P(\lambda) = (\mathbf{M}_{A}[\Lambda_k](\lambda)^\star\otimes I_n)(\lambda M_1+M_0)(\Lambda_k(\lambda)\otimes I_n)$.
 Then:
\begin{enumerate}
\item[\rm (a)] $\displaystyle \frac{\| \mathcal{L}(\lambda) \|_F}{\|P(\lambda)\|_F} = \sqrt{\displaystyle \left( \frac{\|\lambda M_1 + M_0\|_F}{\|P(\lambda)\|_F }\right)^2 +
    \frac{4nk}{\|P(\lambda)\|_F^2}} \geq \frac{1}{\sqrt{2\, (k+1)} }$.
\item[\rm (b)] $\displaystyle \|\lambda M_1 + M_0\|_F \, \geq \, \|P (\lambda) \|_F /\sqrt{2\, (k+1)}$.
\end{enumerate}
\end{lemma}
\begin{proof}
The equality in part (a) follows directly from the structure of $\mathcal{L}(\lambda)$ in \eqref{eq:structured-Kron-pencil}. On the other hand Lemma \ref{lemma:normsproducts} and the related property $\|(\mathbf{M}_A[\Lambda_k](\lambda)^\star\otimes I_n)Z(\lambda)\|_F\leq \min\{\sqrt{k+1},\sqrt{t+1}\}\|Z(\lambda)\|_F$ for any matrix polynomial $Z(\lambda)$ of grade $t$, that we have already used in the proof of Lemma \ref{lemma:FIN}, yield $\|P(\lambda)\|_F\leq \sqrt{k+1}\|(\lambda M_1+M_0)(\Lambda_k(\lambda)\otimes I_n)\|_F\leq \sqrt{2(k+1)}\|\lambda M_1+M_0\|_F$.
This proves (b), which implies the inequality in (a).
\end{proof}

The consequences of Theorem \ref{thm:back_error_main_theorem} and Lemma \ref{lemma:quotient} are the same as the consequences of \cite[Lemma 6.24]{minimal_pencils} and \cite[Theorems 6.22 and 6.23]{minimal_pencils} in the backward error analysis in \cite{minimal_pencils}.
If $\|P(\lambda)\|_F \ll 1$, then $C_{P,\mathcal{L}}$ is huge, since $4kn/\|P(\lambda)\|_F^2$ is huge. 
Moreover, from \eqref{eq:C_pl} and part-(b) in Lemma \ref{lemma:quotient}, we see that if $\|P(\lambda)\|_F \gg 1$, then $C_{P,\mathcal{L}}$ is also huge, since $\|\lambda M_1 + M_0 \|_F$ is huge and $\|\mathcal{L}(\lambda) \|_F/\|P(\lambda)\|_F \geq 1/\sqrt{2\, (k+1)}$. 
Therefore, it is necessary to scale $P(\lambda)$ in advance in such a way that $\|P(\lambda)\|_F = 1$ to have a chance $C_{P,\mathcal{L}}$ moderate.
However,  even in this case, $C_{P,\mathcal{L}}$ is large if $\|\lambda M_1 + M_0\|_F$ is large.
Therefore, to guarantee that $C_{P,\mathcal{L}}$ is a moderate number, in addition to scale $P(\lambda)$, one has to consider only structured block Kronecker pencils with $\|\lambda M_1+M_0\|_F\approx \|P(\lambda)\|_F$.

As a consequence of the discussion in the previous paragraph, we finally state the informal Corollary \ref{cor:FINperturbation}, which establishes sufficient conditions for the structurally backward stability  of the solution of SPEPs and SCPEs via structured block Kronecker pencils.
As it was done in its unstructured version \cite[Corollary 6.25]{minimal_pencils}, for the sake of clarity and simplicity any nonessential numerical constant is omitted.
\begin{corollary}\label{cor:FINperturbation}
Let $P(\lambda) = \sum_{i=0}^g P_i \lambda^i \in \mathbb{F}[\lambda]^{n\times n}$ be a structured matrix polynomial with $\|P(\lambda)\|_F = 1$.
 Let $\mathcal{L}(\lambda)$ be a structured block Kronecker pencil as in \eqref{eq:structured-Kron-pencil}, where the matrix $A$ is one of the matrices in Table \ref{table:1} and it is chosen to guarantee that $\mathscr{S}(P)=\mathscr{S}(\mathcal{L})$, with $g = 2k + 1$ and such that
$P(\lambda) = (\mathbf{M}_{A}[\Lambda_k](\lambda)^\star\otimes I_n)(\lambda M_1+M_0)(\Lambda_k(\lambda)\otimes I_n)$, where $\lambda M_1+M_0$ is the (1,1) block of $\mathcal{L}(\lambda)$. 
Let $\Delta \mathcal{L}(\lambda)$ be any pencil with the same size and structure as $\mathcal{L}(\lambda)$ and with $\|\Delta \mathcal{L}(\lambda)\|_F$ sufficiently small. 
If $\|\lambda M_1 + M_0\|_F \approx \|P(\lambda)\|_F$, then $\mathcal{L}(\lambda) + \Delta \mathcal{L}(\lambda)$ is a strong linearization of a matrix polynomial $P(\lambda) + \Delta P(\lambda)$ with grade $g$ and such that
\begin{equation} \label{eq:informalcor}
\frac{\|\Delta P(\lambda)\|_F}{\|P(\lambda)\|_F} \lesssim \,  (k+1)^{3} \,\sqrt{n} \, \,
\frac{\|\Delta \mathcal{L}(\lambda) \|_F}{\|\mathcal{L}(\lambda) \|_F} \quad \mbox{with}  \quad \mathscr{S}(\Delta P)=\mathscr{S}(P).
\end{equation}
In addition, the right minimal indices of $\mathcal{L}(\lambda) + \Delta \mathcal{L}(\lambda)$ are those of $P(\lambda) + \Delta P(\lambda)$ shifted by $k$, and the left minimal indices of $\mathcal{L}(\lambda) + \Delta \mathcal{L}(\lambda)$ are those of $P(\lambda) + \Delta P(\lambda)$ also shifted by $k$. 
\end{corollary}

\begin{remark}\label{remark:FIN}
  We emphasize that Corollary \ref{cor:FINperturbation} can be applied in particular to the non-permuted block-tridiagonal and block-antitridiagonal structure preserving strong linearizations in \cite{Greeks2,Alternating,Palindromic,Skew} (see also Example \ref{ex:tridiagonal} in this paper), since the Frobenius norm is invariant under permutations, permutations preserve strong linearizations and minimal indices, and the structure of the pencils are preserved under $\star$-congruence.
  Therefore, given one of these block-tridiagonal or block-antitridiagonal structure preserving strong linearization and a perturbation of it with the same structure, we can permute both and transform the corresponding perturbation problem into the problem we have solved in this section.
\end{remark}

\begin{remark}
Notice the following rather surprising result.
The constant \eqref{eq:C_pl}, which shows whether or not  solving SPEPs or SCPEs applying a structured backward stable algorithm to a structured block Kronecker pencil is structurally global backward stable, is the same constant (except by the minor change of replacing $g$ by $k+1$) that shows whether or not the (unstructured) backward stability of solving PEPs and CPEs applying a backward stable algorithm (like the QZ algorithm of the staircase algorithm) to a block Kronecker pencil holds (see \cite[Section 6.3]{minimal_pencils}).
\end{remark}

\section{Conclusions}\label{sec:conclusions}

The numerical solution of a structured polynomial eigenvalue problem  is usually performed by embedding the associated structured matrix polynomial into a matrix pencil with the same structure, called a structure-preserving linearization, and then applying well-established algorithms for structured matrix pencils, like the  palindromic-QR algorithm or the structured  versions of the staircase algorithm,  to the linearization.
This approach guarantees that the computed complete eigenstructure is the exact one of a nearby matrix pencil with the same structure as the original matrix polynomial.
However, it has remained an open problem to determine whether or not the computed eigenstructure is the exact one of a nearby structured matrix polynomial.
In this paper, we have solved this problem for a large family of structure-preserving linearizations, i.e., the family of structured block Kronecker linearizations. 
More precisely, we have performed for the first time a rigorous global and structured backward error analysis of  structured complete polynomial eigenvalue problems solved by using structured block Kronecker linearizations, when the associated matrix polynomial has odd degree and is (skew-)symmetric, (anti-)palindromic or alternating.
In order to perform our analysis  for the six  considered structures in an unified way, we have introduced the formalism of $\mathbf{M}_A$-structured matrix polynomials, and the families of $\mathbf{M}_A$-structured strong block minimal bases pencils  and  $\mathbf{M}_A$-structured block Kronecker pencils, which contains as a particular subclass the family of structured Kronecker pencils.
This analysis has allowed us to identify a huge family of structure-preserving linearizations that yield  perfect structured polynomial backward stability in the solution of structured complete polynomial eigenvalue problems.
In particular, this family contains the famous block-tridiagonal and block-antitridiagonal structure preserving strong linearizations presented in \cite{Greeks2,PalindromicFiedler,Alternating,Palindromic,Skew} and the symmetric and skew-symmetric strong linearizations in \cite{PartI}.

\appendix

\section{Proof of Lemma \ref{lemma:min_singular_val}}\label{appendix:proof}
The goal of this appendix is the computation of the minimum singular value of the matrix
\[
T_A = \left[
      \begin{array}{c|c}
        (bF_{kn}-dE_{kn})\otimes I_{kn} &-I_{kn}\otimes E_{kn}\\\hline
        (aF_{kn}-cE_{kn})\otimes I_{kn}&I_{kn}\otimes F_{kn}
      \end{array}
    \right],
\]
when the matrix $A=\left[\begin{smallmatrix} a & b \\ c & d \end{smallmatrix}\right]$ is equal to any of the following matrices
\begin{align*}
&A_1:=\begin{bmatrix}
1 & 0 \\ 0 & 1 
\end{bmatrix}, \,\,
A_2:=\begin{bmatrix}
-1 & 0 \\ 0 & -1 
\end{bmatrix},\,\,
A_3:=\begin{bmatrix}
0 & 1 \\ 1 & 0
\end{bmatrix},\,\,
A_4:=\begin{bmatrix}
0 & -1 \\ -1 & 0
\end{bmatrix}, \\
&A_5:=\begin{bmatrix}
-1 & 0 \\ 0 & 1
\end{bmatrix},\,\, \mbox{and} \,\,
A_6:=\begin{bmatrix}
1 & 0 \\ 0 & -1
\end{bmatrix}.
\end{align*}

We start by reducing the problem of computing the minimum singular value of $T_A$ to the problem of computing the minimum singular value of a matrix $\widehat{T}_A$ with a size much smaller than the size of $T_A$.
First, notice that we may write the matrix $T_A$ as
  \begin{align*}
  \begin{split}
    T_A=&
    \left[
      \begin{array}{c|c}
        (bF_{kn}-dE_{kn})\otimes I_{k}\otimes I_{n} &-I_{kn}\otimes E_{k}\otimes I_{n}\\\hline
        (aF_{kn}-cE_{kn})\otimes I_{k}\otimes I_{n}&I_{kn}\otimes F_{k}\otimes I_{n}
      \end{array}
    \right]\\
     =&
    \left[
      \begin{array}{c|c}
        (bF_{kn}-dE_{kn})\otimes I_k & -I_{kn}\otimes E_{k}\\\hline
         (aF_{kn}-cE_{kn})\otimes I_{k}&I_{kn}\otimes F_{k}
      \end{array}
    \right]\otimes I_{n}=:\widetilde{T}_A\otimes I_{n}.
    \end{split}
  \end{align*}
Thus, we obtain $\sigma_{\mathrm{min}}(T_A) = \sigma_{\mathrm{min}}(\widetilde{T}_A)$.
Then, we perform a perfect shuffle permutation on each block of the matrix $\widetilde{T}_A$ to swap the order of the Kronecker products.
In other words, there exist permutation matrices $S$, $R_1$, and $R_2$ (see, for example, \cite{VanLoan}) such that 
\begin{align*}
    &\left[
      \begin{array}{c|c}
        S&\\\hline
        &S
      \end{array}
    \right]
    \left[
\begin{array}{c|c}
        (bF_{kn}-dE_{kn})\otimes I_k & -I_{kn}\otimes E_{k}\\\hline
         (aF_{kn}-cE_{kn})\otimes I_{k}&I_{kn}\otimes F_{k}
      \end{array}
    \right]
    \left[
      \begin{array}{c|c}
        R_1^{T}&\\\hline
       \phantom{\Big{(}} &  R_2^{T}
      \end{array}
    \right]\\
    =&
 \left[
\begin{array}{c|c}
        I_k\otimes (bF_{kn}-dE_{kn}) & -E_{k}\otimes I_{kn}\\\hline
         I_{k}\otimes (aF_{kn}-cE_{kn}) & F_{k}\otimes I_{kn}
      \end{array}
    \right] \\
    =&
    \left[
      \begin{array}{c|c}
        I_k\otimes (bF_{k}-dE_{k}) & -E_{k}\otimes I_{k}\\\hline
         I_{k}\otimes (aF_{k}-cE_{k}) & F_{k}\otimes I_{k}
      \end{array}
    \right]\otimes I_{n}=:\widehat{T}_A\otimes I_{n}.
\end{align*}
Therefore, we obtain $\sigma_{\mathrm{min}}(T_A) = \sigma_{\mathrm{min}}(\widetilde{T}_A)=\sigma_{\mathrm{min}}(\widehat{T}_A)$.

We will denote by $\widehat{T}_{i}$ the matrix $\widehat{T}_A$ when $A=A_i$, for $i=1,\hdots,6$.
The rest of the proof consists in showing that the minimum singular value of the matrix $\widehat{T}_i$, for $i=1,\hdots,6$, is equal to the minimum singular value of the matrix
\begin{equation}\label{eq:That}
\widehat{T}:=
\begin{bmatrix}
I_k\otimes E_k & E_k \otimes I_k \\
I_k\otimes F_k & F_k \otimes I_k
\end{bmatrix},
\end{equation}
which, by \cite[Lemmas 6.4 and B.1]{minimal_pencils}, is equal to $2\sin(\pi/(4k))$.

First, notice the following equalities
\begin{align*}
\widehat{T}=&
\begin{bmatrix}
-I_{k^2} & 0 \\ 0 & I_{k^2}
\end{bmatrix}
\left[ \begin{array}{c|c}
-I_k\otimes E_k & -E_k \otimes I_k \\ \hline
I_k\otimes F_k & F_k \otimes I_k
\end{array}\right]  \\
=&\begin{bmatrix}
-I_{k^2} & 0 \\ 0 & I_{k^2}
\end{bmatrix}
\left[ \begin{array}{c|c}
I_k\otimes E_k & -E_k \otimes I_k \\ \hline
-I_k\otimes F_k & F_k \otimes I_k
\end{array}\right]
\begin{bmatrix}
-I_{k(k+1)} & 0 \\ 0 & I_{k(k+1)}
\end{bmatrix}.
\end{align*}
Thus, we immediately obtain $\sigma_{\mathrm{min}}(\widehat{T}_1)=\sigma_{\mathrm{min}}(\widehat{T}_2)=\sigma_{\mathrm{min}}(\widehat{T})=2\sin(\pi/(4k))$.
In addition, notice
\begin{align*}
&\left[ \begin{array}{c|c}
-I_k\otimes F_k & -E_k \otimes I_k \\ \hline
I_k\otimes E_k & F_k \otimes I_k
\end{array}\right] =
\left[ \begin{array}{c|c}
I_k\otimes F_k & -E_k \otimes I_k \\ \hline
-I_k\otimes E_k & F_k \otimes I_k
\end{array}\right]\begin{bmatrix}
-I_{k(k+1)} & 0 \\ 0 & I_{k(k+1)}
\end{bmatrix}, \quad \mbox{and} \\
&\left[ \begin{array}{c|c}
-I_k\otimes E_k & -E_k \otimes I_k \\ \hline
-I_k\otimes F_k & F_k \otimes I_k
\end{array}\right] =
\left[ \begin{array}{c|c}
I_k\otimes E_k & -E_k \otimes I_k \\ \hline
I_k\otimes F_k & F_k \otimes I_k
\end{array}\right]\begin{bmatrix}
-I_{k(k+1)} & 0 \\ 0 & I_{k(k+1)}
\end{bmatrix},
\end{align*}
so we obtain $\sigma_{\mathrm{min}}(\widehat{T}_3)=\sigma_{\mathrm{min}}(\widehat{T}_4)$ and $\sigma_{\mathrm{min}}(\widehat{T}_5)=\sigma_{\mathrm{min}}(\widehat{T}_6)$, and, therefore, we only need to compute $\sigma_{\mathrm{min}}(\widehat{T}_3)$ and $\sigma_{\mathrm{min}}(\widehat{T}_5)$.

Let us compute first the minimum singular value of $\widehat{T}_3$. 
Recall that the singular values of $\widehat{T}_3$ are equal to the square roots of the eigenvalues of the matrix $\widehat{T}_3\widehat{T}_3^T$.
The matrix $\widehat{T}_3\widehat{T}_3^T$ is equal to
\[
\widehat{T}_3\widehat{T}_3^T=\begin{bmatrix}
2I_{k^2} & -\widehat{W}_{k,k} \\ -\widehat{W}_{k,k}^T & 2I_{k^2}
\end{bmatrix} = 2I_{2k^2} -\begin{bmatrix}
0 & \widehat{W}_{k,k} \\ \widehat{W}_{k,k}^T &0
\end{bmatrix},
\]
where $\widehat{W}_{k,k}=I_k\otimes F_kE_k^T+E_kF_k^T\otimes I_k$. 
It is well known  that the eigenvalues of the matrix 
\[
\begin{bmatrix}
 0 & \widehat{W}_{k,k} \\ \widehat{W}_{k,k}^T & 0
\end{bmatrix}
\]
 are $\pm \sigma_1 (\widehat{W}_{k,k}), \ldots, \pm \sigma_{k^2} (\widehat{W}_{k,k})$, where   $\sigma_1 (\widehat{W}_{k,k}) \geq \cdots \geq \sigma_{k^2} (\widehat{W}_{k,k})$ are the singular values of $\widehat{W}_{k,k}$. 
Therefore, the eigenvalues of $\widehat{T}_3 \widehat{T}_3^T$ are $2 \pm \sigma_1 (\widehat{W}_{k,k}), \ldots, 2 \pm \sigma_{k^2} (\widehat{W}_{k,k})$, which implies \[
\sigma_{\mathrm{min}}(\widehat{T}_3)=\sqrt{2-\sigma_{\mathrm{max}}(\widehat{W}_{k,k})}.
\]
So we have to compute the largest singular value of the matrix $\widehat{W}_{k,k}$.
To this aim, let us denote by $R_k$ the $k\times k$ reverse identity matrix, i.e., the matrix
\[
R_k:=\begin{bmatrix}
& & 1 \\ & \iddots & \\ 1 
\end{bmatrix}\in\mathbb{R}^{k\times k}.
\]
Notice $R_kF_kE_k^TR_k=E_kF_k^T$.
Thus, we have $(I_k\otimes R_k)\widehat{W}_{k,k}(I_k\otimes R_k)= I_k\otimes E_kF_k^T+E_kF_k^T\otimes I_k=:W_{k,k}$.
From \cite[Proposition B4]{minimal_pencils} and the previous argument, we get $\sigma_{\mathrm{max}}(\widehat{W}_{k,k})=\sigma_{\mathrm{max}}(W_{k,k})=2\cos(\pi/(2k))$.
Therefore, using a simple trigonometric identity, we obtain 
\[
\sigma_{\mathrm{min}}(\widehat{T}_3)=\sqrt{2-2\cos\left(\frac{\pi}{2k}\right)} = 2\sin\left(\frac{\pi}{4k}\right),
\]
which is the desired result.

Let us compute, now, the minimum singular value of the matrix $\widehat{T}_5$. 
To this purpose, note that
\[
\begin{bmatrix}
-I_{k^2} & \\ & I_{k^2} 
\end{bmatrix}
\widehat{T}_5 =
\begin{bmatrix}
I_k\otimes E_k & E_k\otimes I_k \\
I_k\otimes (-F_k) & F_k\otimes I_k
\end{bmatrix} =: \widetilde{T}_5,
\]
and that $\widehat{T}_5$ and $\widetilde{T}_5$ have the same singular values.
If we define the diagonal matrices $S_k:=\diag ((-1)^0,(-1)^1,  \hdots,(-1)^{k-1})\in\mathbb{R}^{k\times k}$ and $S_{k+1}:= \diag(S_k,(-1)^k)\in\mathbb{R}^{(k+1)\times (k+1)}$, then $S_kE_kS_{k+1}=E_k$ and $S_kF_kS_{k+1}=-F_k$.
Therefore
\[
\begin{bmatrix}
I_k\otimes S_k \\ & I_k\otimes S_k
\end{bmatrix}
\widetilde{T}_5
\begin{bmatrix}
I_k\otimes S_{k+1}\\&I_{k+1}\otimes S_k
\end{bmatrix}=\widehat{T},
\]
where $\widehat{T}$ is the matrix in \eqref{eq:That}.
So the singular values of $\widehat{T}_5$ and $\widehat{T}$ coincide and $\sigma_{\mathrm{min}}(\widehat{T}_5)=\sigma_{\mathrm{min}}(\widehat{T})=2\sin(\pi/(4k))$, which completes the proof.

\end{document}